\documentclass{article}
\usepackage{subcaption}
\usepackage{amsmath}
\numberwithin{equation}{section}
\usepackage{cite}
\usepackage{amssymb, xcolor, amsthm} 
\usepackage{tikz}
\usepackage[utf8]{inputenc}
\usepackage[english]{babel}
\usepackage{graphicx}
\usepackage{epsfig}
\usepackage{mathtools}
\usepackage{datetime}
\usepackage{hyperref}
\usepackage[utf8]{inputenc}
\usepackage[english]{babel}
\usepackage{amsfonts}

 \DeclareMathOperator\intr{int\,}
\def\j{{\mathbf{j}}}
\def\n{{\mathbf{n}}}
\usepackage{blindtext}
\usepackage[english]{babel}
\newtheorem{theorem}{Theorem}[section]
\newtheorem{corollary}{Corollary}[theorem]

\newtheorem{proposition}{Proposition}[section]
\newtheorem{remark}{Remark}
\theoremstyle{definition}
\newtheorem{definition}{Definition}[section]
\newenvironment{example*}
  {\addtocounter{theorem}{-1}\example}
  {\endexample}
  \newcounter{example}[section]
\newenvironment{example}[1][]{\refstepcounter{example}\par\medskip
   \noindent \textbf{Example~\theexample. #1} \rmfamily}{\medskip}
\newcommand{\be}{\begin{equation}}
\newcommand{\ee}{\end{equation}}
\newcommand{\bes}{\begin{equation*}}
\newcommand{\ees}{\end{equation*}}

\newcommand{\N}{\mathbb{N}}

\newcommand{\R}{\mathbb{R}}

\newcommand{\cI}{\mathcal{I}}

\newcommand{\nl}{\vskip 10pt\noindent}

\newcommand{\wt}[1]{{\widetilde{#1}}}
\DeclareMathOperator*\lowdim{\underline{dim}_B}
\DeclareMathOperator*\updim{\overline{dim}_B}
\DeclareMathOperator*\lowlim{\underline{lim}}
\DeclareMathOperator*\uplim{\overline{lim}}
\DeclareMathOperator\Lip{Lip}

\newcommand{\abs}[1]{\left\vert{#1}\right\vert}

\newcommand{\vertiii}[1]{{\left\vert\kern-0.25ex\left\vert\kern-0.25ex\left\vert #1 
    \right\vert\kern-0.25ex\right\vert\kern-0.25ex\right\vert}}

\title{Multivariate  Zipper Fractal Function}

\begin{document}
\begin{center}
{\Large \bf {Multivariate  Zipper Fractal Functions}}
\end{center}
\begin{center}
{{ D. Kumar,  A. K. B. Chand, P. R. Massopust }}\\
Department of Mathematics\\
Indian Institute of Technology Madras\\
Chennai - 600036, India, \\
School of Computation, Information and Technology\\
Department of Mathematics\\
Technical University of Munich (TUM)\\
85748 Garching b. M\"unchen, Germany\\
Email: deependra030794@gmail.com; chand@iitm.ac.in; massopus@ma.tum.de
\end{center}

\begin{abstract}
A novel approach to zipper fractal interpolation theory for functions of several variables is presented.  Multivariate zipper fractal functions are constructed and then perturbed through free choices of base functions, scaling functions, and a binary matrix called  signature to obtain their zipper $\alpha$-fractal versions. In particular, we propose a multivariate Bernstein zipper fractal function and study its approximation properties such as shape preserving aspects, non-negativity, and coordinate-wise monotonicity. In addition, we derive bounds for the graph of a multivariate zipper fractal function by imposing conditions on the scaling factors and the H\"older exponent of the associated germ function and base function. The Lipschitz continuity of multivariate Bernstein functions is also studied in order to obtain estimates for the box dimension of multivariate Bernstein zipper fractal functions.

\end{abstract}


\section{Introduction}
Interpolation is a basic and fundamental subject in numerical analysis and approximation theory for the continuous representation of discrete data. 
A standard way to obtain a bivariate interpolation from univariate interpolation functions is by using a tensor product if the underlying two variables are considered separately. This procedure is also adapted to multivariate interpolation when data from a multivariate function are prescribed on a Cartesian product of grid points. 
There are numerous ways to approximate  multivariate functions by using, for instance, multivariate polynomials\cite{Trefethen, Mond, Derriennic}, splines\cite{Schultz,Wang_Tan,Nielson}, tensor products splines\cite{Jonge_Zanten}, local methods, global methods, blending-function methods\cite{Gordon}, and Hermite Interpolation Formulas\cite{A_Spitzbart}. All these methods may have advantages
and disadvantages depending on the nature of the data and the application. When data is generated from a very irregular multivariate function, the above methods are not ideal to provide a deep understanding of the true multivariate features. This paper proposes a new approach to describe non-linear patterns associated with a multivariate data generating function by means of zipper multivariate fractal interpolation functions (FIFs). 

Fractal surfaces continue to draw attention to scientists and engineers due to their useful applications in various areas such as medical sciences, surface physics, chemistry, bio-engineering,  metallurgy, computer science, electrical engineering, and earth science.  Fractal surfaces have been found  to be good approximations of natural surfaces in these areas because of their special properties, such as self-similarity, visualization at different scales, and a non-integral fractal dimension. The construction of fractal surfaces using iterated function systems (IFSs) with co-planar boundary were first introduced by Massopust in \cite{Massopust2} with different scaling factors. The construction of fractal surfaces with arbitrary boundary values but equal scaling factors was taken up by Geronimo and  Hardin in \cite{Geronimo_Hardin}.  Hardin and Massopust investigated more general fractal functions defined on complexes of simplices  $D\subseteq \mathbb{R}^n$ into $\mathbb{R}^m$ in \cite{Hardin_Massopust}. Bouboulis and Dalla \cite{Bouboulis_Dalla2,Bouboulis_Dalla3} constructed fractal surfaces using IFSs over grids or rectangular domains. Using tensor product of cardinal spline,
Chand and Navascue\'s proposed bicubic fractal surfaces  \cite{Chand_Navascues}.
 The theory of fractal surfaces has been investigated along various directions, for instance, \cite{Massopust,Ruan_Xu,Bouboulis_Dalla1,Chand_Kapoor,BEHM}. The shape preserving fractal surfaces are developed recently using blending functions and
 univariate fractal functions, see for instance \cite{Chand_vij,Chand_vv, Chand_kr}.
 Aseev \cite{VAO} introduced  the construction of fractals by using the
 idea of a zipper, where the entire graph can be mapped to two consecutive nodes in two different ways. Subsequently, the theory of multi-zipper was investigated  by Tetenov et. al \cite{ATK}. Introducing such a binary array called the signature of a zipper, the class of affine zipper FIFs was introduced recently into the literature by Chand et al. \cite{CVVT}. Further, the calculus of zipper FIFs and cubic zipper FIFs  are studied by  Reddy \cite{kmr} and the approximation by smooth zipper fractal
 functions is investigated in \cite{Vijay_Vijender_Chand}. 
 
 In this paper, we introduce
 the concept of multivariate zipper fractal interpolation functions (ZFIFs) 
 to interpolate and approximate multivariate data  or a multivariate function by using a suitable
 binary matrix called zipper  $\epsilon$. These multivariate zipper fractal functions 
 are more general  than the existing classical and fractal approximants. 
 Based on the existence of ZFIFs, we construct a novel class of multivariate Bernstein zipper $\alpha$-fractal functions using Bernstein polynomials  $B_{n_1,...,n_m}f$\cite{Foupouagnigni_Wouodjie,Davis} as base functions in its IFS  
  for a given $f\in C\left(\prod\limits_{k=1}^mI_k\right)$, where the $I_k$ are compact intervals in $\R$. Multivariate Bernstein zipper $\alpha$-fractal functions $f^{\alpha,\epsilon}_{n_1n_2...n_m}$ converge uniformly to $f$  as $n_i \to \infty $ for all $i$, without having to alter the scaling functions. We prove that the multivariate Bernstein polynomial  $B_{n_1,...,n_m}f$ is Lipschitz if $f$ is.
Employing the H\"older exponents of the germ function, the base function, and the scaling factors, we derive bounds for box-dimension of the graph of a multivariate zipper $\alpha$-fractal function. Our results are more general than several existing results in univariate and multivariate cases \cite{Vijender, Akhtar_Prasad, Pandey_vishwanathan}.
 
This paper is organized as follows. Section \ref{1sec} introduces the basics of univariate zipper fractal functions including its construction. Section \ref{2sec2} is concerned with the constructive existence of  multivariate zipper fractal interpolation on a given multivariate  data set by means of a binary signature matrix.  In addition,  a multivariate (germ) function is fractalized using a zipper setting to present its fractal version through a suitable base function.  When the base function is taken to be a multivariate Bernstein function we obtain multivariate Bernstein zipper $\alpha$-fractal functions. These together with some of their approximation-theoretic properties are introduced in Section \ref{2sec3}. 
Placing restrictions on the scaling factors, it is shown in Section \ref{2sec4} that multivariate Bernstein zipper $\alpha$-fractal functions preserve the non-negativity of multivariate germ functions. This is then extended to coordinate-wise monotonicity in Section \ref{2sec6}. Finally, we derive bounds   for the box dimension of the graph of a multivariate zipper $\alpha$-fractal function. Similar bounds are obtained for multivariate zipper Bernstein  $\alpha$-fractal functions under the assumption that $B_{\textbf{n}}f$ is H\"olderian given that $f$ is.
\section{Basics of Zipper Fractal Functions}\label{1sec}
In this section, we discuss the basics of IFSs and zippers and present the construction of zipper fractal functions. More details can be found in \cite{VAO,Barnsley,CVVT}.

In the following, for an $m\in \N$, we denote by $\N_m :=\{1, 2, \ldots, m\}$ the initial segment of $\N$ of length $m$.
\begin{definition}
Let $1< N\in\N$ and let $w_i :X \to X$, $i\in \N_{N-1}$, be non-surjective maps on a complete metric space $(X,d)$. Then, the system $\widetilde{I} :=\{X; w_i,i\in \N_{N-1}\}$ is called an  IFS with vertices $\{k_1,k_2,\dots,k_N\}\subset X$ provided that
\[
w_i(k_1)=k_{i}\quad\text{and}\quad w_i(k_N)=k_{i+1}. 
\]
The points $k_1$ and $k_N$ are called the initial and final point of the IFS, respectively.
\end{definition}
\begin{definition}
For a binary vector $ \epsilon:=(\epsilon_1,\epsilon_2,\dots,\epsilon_{N-1}) \in \{0,1\}^{N-1}$ called signature, let $w_i: X\to X$, $i\in \N_{N-1}$, be non-surjective maps on a complete metric space $(X,d)$ such that  $w_i$ satisfies 
\[
w_i(k_1)=k_{i+\epsilon_i}\quad\text{and}\quad w_i(k_N)=k_{i+1-\epsilon_i} 
\]
for a given set $\{k_1,k_2,\dots,k_N\}\subset X$.

Then, the system $\widetilde{I}=\{X;w_i, i\in \N_{N-1}\}$ is called a zipper with vertices $\{k_1,k_2,$ $\dots,k_N\}$. Any non-empty compact set $A \subset X$ satisfying the self-referential equation
\begin{equation*}
A=\bigcup^{N-1}_{i=1} w_i(A), 
\end{equation*}
is called the attractor or zipper fractal corresponding to the zipper $\widetilde{I}$.
\end{definition}
Clearly, an IFS is a particular case of a zipper when the signature satisfies  $\epsilon_i=0$, for all $i \in \mathbb{N}_{N-1}$. 

Next, we will review the construction of zipper FIFs (ZFIFs) from
a suitable zipper which is constructed from a given set of interpolation data.
\vskip 6pt
Suppose $2 < N\in \N$. Let a set of interpolation points $\{(x_i,y_i) \in I \times \mathbb{R}: i\in \mathbb{N}_N\}$ be given where $x_1<x_2<\dots<x_N$ is a partition of the interval $I:=[x_1,x_N]$ and  $y_i \in [c,d] \subset \mathbb{R}$, $\forall i \in \mathbb{N}_N$. Let us set $I_i:=[x_i,x_{i+1}]$ and $ D:=I \times [c,d]$. Let $u_i^{\epsilon}:I \rightarrow I_i$, $i\in \N_{N-1}$, be contractive homeomorphisms such that
\begin{equation}\label{eqn1}
u_i^{\epsilon}(x_1)=x_{i+\epsilon_i} \quad\text{and}\quad u_i^{\epsilon}(x_N)=x_{i+1-\epsilon_i}.
\end{equation}
Note that if $u_i^{\epsilon}(x) :=a_ix+b_i$ and $\epsilon_i=1$, then the horizontal scaling factors $a_i$ can be negative. 

Define  $v_i^{\epsilon}: D \rightarrow \mathbb{R}$, $i \in \mathbb{N}_{N-1}$, by 
\begin{equation*}
v_i^{\epsilon}(x,y):=\alpha_i(x)y+q_i(x), 
\end{equation*}
where $\alpha_i$ and $q_i$ are continuous functions on $I$ such that $\|\alpha_i\|_{\infty}<1$, and 
\begin{equation}\label{eqn2}
v_i^{\epsilon}(x_1,y_1)=y_{i+\epsilon_i},\quad\text{and}\quad  v_i^{\epsilon}(x_N,y_N)=y_{i+1-\epsilon_i}, \quad  i \in \mathbb{N}_{N-1}.
\end{equation}
Here $v_i^{\epsilon}$ either contracts  or flips the graph of $f$ over $I$ to $I_i$. Using these maps, we define maps
 $w_i: D \rightarrow I_i \times \mathbb{R}$, $i\in\N_{N-1}$, by
\begin{equation*}
w_i^{\epsilon}(x,y) :=(u_i^{\epsilon}(x),v_i^{\epsilon}(x,y)), \; \; \forall (x,y) \in D.
\end{equation*}
The zipper IFS for the construction of ZFIFs is then given by 
\[
\widetilde{I}^{\epsilon} : =\{D;w_i^{\epsilon}, i\in\N_{N-1}\}
\]  
with vertices $\{v_i=(x_i,y_i)\}_{i=1}^N$ and signature $\epsilon=\{\epsilon_1,\epsilon_2,\dots,\epsilon_{N-1}\}$. For more details, please consult \cite{CVVT}. 
\begin{theorem}\label{2ta} 
The above zipper $\widetilde{I}^{\epsilon}= \{D;w_i^{\epsilon},i\in\N_{N-1}\}$ enjoys the following properties.
\begin{enumerate}
\item[(i)] There exists a unique non-empty compact set $G\subset K$ such that
\[
G=\bigcup\limits^{N-1}_{i=1} w_i^{\epsilon}(G).
\]
\item[(ii)] $G$ is the graph of a continuous function $f^{\epsilon}: I \rightarrow \mathbb{R}$ which interpolates the data $\{(x_i,y_i):i\in\N_N\}$, i.e., $G=\{(x,f^{\epsilon}(x) : x \in I\}$ and, for $i\in\N_N$, $f^{\epsilon}(x_i)=y_i$.
\end{enumerate}
\end{theorem}
The above theorem shows the existence of a zipper interpolation function whose graph is the attractor of an associated zipper IFS. 

To obtain a recursive formula for the ZFIF  $f^{\epsilon}$, we
proceed as follows. Let $\epsilon \in \{0,1\}^{N-1}$ be fixed, and let 
\[
\wt{C}(I):=\{g \in C(I) : g(x_1)=y_1, \; g(x_N)=y_N\}. 
\]
Then, $\wt{C}(I)$ is a closed \emph{metric} subspace of $C(I)$ and complete with respect to the metric $d$ induced by the sup-norm.

Now define a Read-Bajraktarevi\'c operator $T:\wt{C}(I) \rightarrow \wt{C}(I)$ by 
\begin{equation*}
(Tg)(x) := \sum_{i=1}^{N-1} v_i^{\epsilon} ( (u_i^{\epsilon})^{-1} (x),
g \circ (u_i^{\epsilon})^{-1} (x)) \; \chi_{u_i^{\epsilon}(I)} (x), \quad x \in I.
\end{equation*}
Clearly,  as $\|\alpha_i\|_{\infty}<1$, $T$ is contraction on $(\wt{C}(I),d)$. By the Banach fixed point theorem, $T$ has a unique fixed point $f^{\epsilon}$ which obeys the self-referential equation
\begin{equation*}
f^{\epsilon} = \sum_{i=1}^{N-1} v_i^{\epsilon} ( (u_i^{\epsilon})^{-1} ,
f^{\epsilon} \circ (u_i^{\epsilon})^{-1} ) \; \chi_{u_i^{\epsilon}(I)}.
\end{equation*}
We call this interpolating function $f^{\epsilon}$ a zipper fractal interpolation function (ZFIF) corresponding to the given data $\{(x_i,y_i):i\in\N_N\}$  and the signature $\epsilon=(\epsilon_1,\epsilon_2,\dots,\epsilon_{N-1})\in \{0,1\}^{N-1}$ for a fixed scaling function vector $\alpha:=(\alpha_1,\alpha_2,\dots,\alpha_{N-1})$.  

For a prescribed function $f \in C(I)$, if we choose 
\[
q_i(x):=f(u_i(x))-\alpha_i(x)b(x), 
\]
for $i \in \mathbb{N}_{N-1}$, and $y_i=f(x_i)$, for $i \in \mathbb{N}_N$, where $b$ is called a base function satisfying $f(x_1)=b(x_1)$ and $f(x_N)=b(x_N)$, then the corresponding ZFIF $f_{\alpha}^{\epsilon}$ is called a zipper $\alpha$-fractal function. The concept of such zipper fractal functions will be extended to the multivariate setting in the next section.
\section{Multivariate Zipper Fractal Functions}\label{2sec2}
In the first part of this section, we show the existence of multivariate ZFIFs in a deterministic way with constant scaling functions. This concept is then used to perturb any multivariate function $f$ to construct its fractal analogue by using a suitable base function in the second part.
\subsection{Multivariate Zipper Fractal Interpolation }\label{2subsec2.1}
For $m\in \mathbb{N}$, we adopt the following notation. 
\begin{gather*}
\mathbb{N}_{m,0} := \{0,1,...,m\},\quad \partial\mathbb{N}_{m,0} :=\{0,m\},  \quad \intr{\mathbb{N}_{m,0}} :=\{1,...,m-1\}.
\end{gather*}
Finite tuples of elements from $\N_m$ are denoted by expressions like $\j:=(j_1,\cdots,j_m)$. Furthermore, let
\begin{gather*} 
\epsilon :=(\epsilon^1,...,\epsilon^m) \in \prod_{k=1}^{m}\{0,1\}^{\N_k},\\
 \cI := \prod_{k=1}^m I_k,\quad I_k := [a_k, b_k]\subset\R, \;a_k < b_k,\; k \in \N_m,
\end{gather*}
where $\prod$ denotes the Cartesian product of sets. 

Let $2\leq m\in\N$ and let $C(\mathcal{I})$ denote the Banach space of continuous functions $f:\cI\to\R$ equipped with the sup-norm. For each $k\in \N_m$, define a partition of $I_k$ by
\[
a_k =: x_{k,0}< \cdots <x_{k,N_k}:=b_k.    
\]
Consider the set of interpolation data points
\begin{center}
$\Delta := \left\{(x_{1,j_1},...,x_{m,j_m},y_{\j})\in \cI \times\R:  \mathbf{j} \in \prod\limits_{k=1}^m \mathbb{N}_{k,0}\right\}$.    
\end{center}
Since $\{a_k=x_{k,0},...,x_{k,N_k}=b_k\}$ is the partition of $I_k$, denote the $j_k$-th sub-interval of $I_k$ by $I_{k,j_k}=[x_{k,j_k-1},x_{k,j_k}]$, $j_k\in \mathbb{N}_{N_k}$. For every $ j_k\in \mathbb{N}_{N_k}$, consider an affine map 
$u_{k,j_k}^{\epsilon^k}:I_k \to I_{k,j_k}$ satisfying 
\begin{equation}\label{2eq1}
|u_{k,j_k}^{\epsilon^k}(x)-u_{k,j_k}^{\epsilon^k}(x')|\leq \alpha_{k,j_k}|x-x'|, \quad\forall x, x' \in I_k,
\end{equation} 
where $0\leq \alpha_{k,j_k} < 1$, and
\begin{equation}\label{2eq2}
\begin{cases}
    u_{k,j_k}^{\epsilon^k}(x_{k,0})=x_{k,j_k-1+ \epsilon^k_{j_k}}~ \text{and} ~ u_{k,j_k}^{\epsilon^k}(x_{k,N_k})=x_{k,j_k- \epsilon^k_{j_k}}, \text{if}~ j_k ~\text{is  odd},\\
    u_{k,j_k}^{\epsilon^k}(x_{k,0})=x_{k,j_k-\epsilon^k_{j_k}}~ \text{and}~ u_{k,j_k}^{\epsilon^k}(x_{k,N_k})=x_{k,j_k-1+ \epsilon^k_{j_k}}, \text{if}~ j_k ~\text{is  even}.
    \end{cases}
\end{equation}
From \eqref{2eq2},  it is easy to check that 
\begin{equation}\label{2eq3}
    (u_{k,j_k}^{\epsilon^k})^{-1}(x_{k,j_k})= (u_{k,j_k+1}^{\epsilon^k})^{-1}(x_{k,j_k}), \forall j_k \in \intr\mathbb{N}_{N_k,0}.
\end{equation}
For each $k\in \mathbb{N}_{m}$,  define a map $\tau_k: \mathbb{N}_{N_k}\times\{0,N_k\}\to \mathbb{Z}$ by
\begin{equation}\label{2eq4}
\begin{cases}
\tau_k(j,0):=j-1+\epsilon_{j}^k \quad\text{and}\quad \tau_k(j,N_k):=j-\epsilon_j^k \;\text{if}\; j \;\text{is odd},\\
\tau_k(j,0):=j-\epsilon_j^k \quad\text{and}\quad \tau_k(j,N_k):=j-1+\epsilon_j^k \;\text{if}\; j \;\text{is even}.
   \end{cases}
\end{equation}
      Using  \eqref{2eq4}, we can rewrite \eqref{2eq2} as
    \begin{equation}\label{2eq5}
       u_{k,j_k}^{\epsilon^k}(x_{k,i_k})=x_{k,\tau_k(j_k,i_k)}, \quad\forall j_k \in \mathbb{N}_{N_k},\; i_k\in \partial\mathbb{N}_{N_k,0},\; k\in \mathbb{N}_m.
    \end{equation}  
Let $\mathcal{K}:=\mathcal{I}\times \mathbb{R}$. For each $\mathbf{j}\in \prod\limits_{k=1}^{m}\mathbb{N}_{N_k}$, define a continuous function $v_{\mathbf{j}}^{\epsilon}:\mathcal{K} \to \mathbb{R}$ satisfying the following conditions:
\begin{equation}\label{2eq6}
    v_{\mathbf{j}}^{\epsilon}(x_{1,i_1},...,x_{m,i_m},y_{i_1...i_m})=y_{\tau_1(j_1,i_1)...\tau_m(j_m,i_m)}, \quad\forall \; \mathbf{i} \in \prod_{k=1}^{m}\partial \mathbb{N}_{N_k,0}  
\end{equation} and
\begin{equation}\label{2eq7}
    | v_{\mathbf{j}}^{\epsilon}(x_1,...,x_m,y)- v_{\mathbf{j}}^{\epsilon}(x_1,...,x_m,y')|\leq \gamma_{\mathbf{j}}|y-y'|,
\end{equation}
for all $(x_1,...,x_m)\in \mathcal{I}$ and   $y,y'\in \mathbb{R}$,
where $0\leq \gamma_{\mathbf{j}} <1$.

Next, for any $\mathbf{j}\in \prod\limits_{k=1}^{m} \mathbb{N}_{N_k}$,  we define $W_{\mathbf{j}}^{\epsilon}:\mathcal{K}\to\mathcal{K} $ by 
\begin{equation}\label{2eq8}
    W^{\epsilon}_{\mathbf{j}}(x_1,...,x_m,y):=(u_{1,j_1}^{\epsilon^1}(x_1),...,u_{m,j_m}^{\epsilon^m}(x_m), v_{\mathbf{j}}(x_1,...,x_m,y)).
\end{equation}
The system
\begin{equation}\label{2eq9}
I^{\epsilon} = \left\{ K,W_{\mathbf{j}}^{\epsilon}: {\mathbf{j}}\in \prod\limits_{k=1}^m \mathbb{N}_k\right\}.
\end{equation} is called multi-zipper IFS with vertices $ \Delta= \{(x_{k,j_1},...,x_{k,j_m},y_j): {\mathbf{j}} \in \prod\limits_{k=1}^m \mathbb{N}_k\}$ and signature $\epsilon$.
Let us consider 
\[\mathcal{G}=\left\{g\in C(\mathcal{I}): g(x_{1,j_1},...,x_{m,j_m})=y_{\mathbf{j}},~ \forall ~{\mathbf{j}}\in \prod\limits_{k=1}^{m}\partial\mathbb{N}_{N_k,0}\right\}\] 
endowed with  the uniform metric   
\[\rho(f,g)=\max\left\{ |f(x_1,...,x_m)-g(x_1,...,x_m)|:(x_1,...,x_m)\in \prod\limits_{k=1}^{m}I_k\right\}\] 
for $f,g\in \mathcal{G}$.
Then $(\mathcal{G},\rho)$ is complete metric space. 

Define a Read-Bajrakterivi\`c operator $T^{\epsilon}:\mathcal{G}\to \mathcal{G}$  on $(\mathcal{G},\rho)$\cite{Massopust} by 
\begin{align}\label{2eq10}
T^{\epsilon}g(x) := & \sum_{\mathbf{j} \in \prod\limits_{k=1}^{m} \mathbb{N}_{N_k}} v_{\mathbf{j}}^{\epsilon}((u_{1,j_1}^{\epsilon^1})^{-1}(x_1),...,(u_{m,j_m}^{\epsilon^m})^{-1}(x_m),\nonumber\\ 
& \qquad g((u_{1,j_1}^{\epsilon^1})^{-1}(x_1),...,(u_{m,j_m}^{\epsilon^m})^{-1}(x_m))) \chi_{u_{\mathbf{j}}^{\epsilon}(\mathcal{I})}(x),
\end{align}
for all $x:=(x_1,...,x_m)\in \mathcal{I}$.

One observes that $T^{\epsilon}g$ is not continuous for all $\epsilon$. In order to achieve continuity, we restrict the signature $\epsilon$ to $\epsilon_{j_k}^k=\epsilon_{j_k+1}^k$, for each $j_k\in\N_{{N}_k-1}$, where $\epsilon_{j_k}^k$ denotes the $j_k$-th component of the binary column vector $\epsilon^k$.
\begin{theorem}\label{2th0}
Let $\Delta := \left\{(x_{1,j_1},...,x_{m,j_m},y_{\mathbf{j}}): \j \in \prod\limits_{k=1}^m N_{k,0}\right\}$ be a set of multivariate interpolating data points and  
$\epsilon=(\epsilon^1,...,\epsilon^m) \in \prod\limits_{k=1}^{m}\{0,1\}^{N_k}$ be a signature for the  IFS $I^{\epsilon}=\left\{ \mathcal{K},W_{\j}^{\epsilon}:\j\in \prod\limits_{k=1}^{m}\mathbb{N}_{N_k}\right\}$ as defined in \eqref{2eq9}. Assume that
for all   $j_k \in $  $\intr\mathbb{N}_{N_k,0}$, $1\leq k \leq m$,
\begin{equation}\label{2eq0}
\begin{split}
(u_{k,j_k}^{\epsilon^k})^{-1}(x_{k,j_k}) & = (u_{k,j_k+1}^{\epsilon^k})^{-1}(x_{k,j_k})=:x_k^{*},\\
    v_{j_1,...,j_k,...,j_m}^{\epsilon}(x_1,...,x_{k-1}, & x_{k}^*,x_{k+1},...,x_m, y)\\   =v_{j_1,...,j_k+1,...,j_m}^{\epsilon}(x_1,...,&x_{k-1},x_{k}^*,x_{k+1},...,x_m, y),  
\end{split}
\end{equation}
\noindent where  $(x_1,...,x_{k-1},x_{k+1},...,x_m)\in  \prod\limits_{i=1,i\ne k}^{m}I_i$, $y \in \mathbb{R}$. Then, there exists a continuous function $f^{\epsilon}:\cI\to\R$  such that
\begin{enumerate}
\item[(i)] $f^{\epsilon}$ interpolates the given multivariate data set $\Delta $, that is, 
\[
 f^{\epsilon}(x_{1,j_1},...,x_{m,j_m})=y_{\j}, \quad\forall \j \in \prod\limits_{k=1}^{m}\mathbb{N}_{N_k,0}.
\]
\item[(ii)] $ G:=\{(x,f^{\epsilon}(x)): x \in \mathcal{I}\}$ is the graph of the Zipper fractal function $f^{\epsilon}$ and satisfies 
\[
G=\bigcup_{\mathbf{j} \in \prod\limits_{k=1}^m \mathbb{N}_{N_k}}W_{\mathbf{j}}^{\epsilon}(G).
\]
\end{enumerate}
\end{theorem}
\begin{proof}

The proof of this theorem is  similar to the bivariate case as explained in \cite{Ruan_Xu}, but for the reader's convenience, we give a short  explanation  of it. 
 
It follows from \eqref{2eq10} that $Tg$ is continuous on  $\prod\limits_{k=1}^{m}I_{k,j_k}$. To prove that  $Tg$ is continuous on the $m$-dimensional hyperrectangle $\prod\limits_{k=1}^{m}I_{k}$, it is sufficient  to show that $Tg$ is well-defined on  the hyperrectangle  $\prod\limits_{k=1}^{m}I_{k,j_k}$.

\noindent
\textbf{Claim:} $T^{\epsilon}$ is well-defined.
\vskip 3pt\noindent
 Assume   $j_k \in \intr\mathbb{N}_{N_k,0}$, $1\leq k \leq m$,
 and $X :=(x_1,\cdots,x_k,\cdots,x_m)\in \mathcal{I}$, with $x_k=x_{k,j_k}$.
Then there are following two cases:
\nl
\textbf{Case (i):} Assume $x_{k,j_k}$ is an element of $I_{k,j_k}$. Then, by \eqref{2eq0}, we have  
 \begin{align*}
 T^{\epsilon}f(X) &= v_{j_1,...,j_k,..,j_m}((u_{1,j_1}^{\epsilon^1})^{-1}(x_1),..,(u_{k,j_k}^{\epsilon^k})^{-1}(x_k),\ldots,\\
    &\qquad(u_{m,j_m}^{\epsilon^m})^{-1}(x_m),g((u_{1,j_1}^{\epsilon^11})^{-1}(x_1),...,(u_{m,j_m}^{\epsilon^m})^{-1}(x_m))). 
 \end{align*}
 \nl
\textbf{Case (ii):} Consider $x_{k,j_k}$ as an element of $I_{k,j_k+1}$. Then, by \eqref{2eq0}, we have  
 \begin{align*}
T^{\epsilon}f(X)= &v_{j_1,...,j_k+1,..,j_m}((u_{1,j_1}^{\epsilon^1})^{-1}(x_1),..,(u_{k,j_k+1}^{\epsilon^k})^{-1}(x_k),\ldots,\\
    &\qquad (u_{m,j_m}^{\epsilon^m})^{-1}(x_m),g((u_{1,j_1}^{\epsilon^11})^{-1}(x_1),...,(u_{m,j_m}^{\epsilon^m})^{-1}(x_m)))\\
   =& v_{j_1,...,j_k,..,j_m}((u_{1,j_1}^{\epsilon^1})^{-1}(x_1),..,(u_{k,j_k}^{\epsilon^k})^{-1}(x_k),\ldots, \\
    &\qquad (u_{m,j_m}^{\epsilon^m})^{-1}(x_m),g((u_{1,j_1}^{\epsilon^11})^{-1}(x_1),...,(u_{m,j_m}^{\epsilon^m})^{-1}(x_m))). 
 \end{align*}
Similarly, we can check the other possible cases. Hence,  $Tg$ is well-defined  on the boundary of  $\prod\limits_{k=1}^{m}I_{k,j_k}$ and therefore continuous on $\mathcal{I}$.

 Let $\mathbf{i}:=(i_1,...,i_m)\in \prod\limits_{k=1}^{m}\mathbb{N}_{N_k,0}$.  Choose $\mathbf{l}:=(l_1,...,l_m)\in \prod\limits_{k=1}^{m} \partial \mathbb{N}_{N_k,0},$ $ \mathbf{j}\in \prod\limits_{k=1}^{m}\mathbb{N}_{N_k}$ such that $\mathbf{i}=(\tau_1(j_1,l_1),...,\tau_m(j_m,l_m))$. According to the definition of $\tau_k$, we have
 $(u_{k,j_k}^{\epsilon^k})^{-1}(x_{k,i_k})=x_{k,l_k}$, for all $k\in \mathbb{N}_{m}$. Using \eqref{2eq6} and \eqref{2eq10}, we obtain
 \begin{align*}
T^{\epsilon}g(x_{1,i_1},...,x_{m,i_m}) &= v_{\textbf{j}}^{\epsilon}((u_{1,j_1}^{\epsilon^1})^{-1}(x_{1,i_1}),..,(u_{m,j_m}^{\epsilon^m})^{-1}(x_{m,i_m}),\\
  &\qquad f((u_{1,j_1}^{\epsilon^1})^{-1}(x_{1,i_1}),...,(u_{m,j_m}^{\epsilon^m})^{-1}(x_{m,i_m}))) \\
  &=v_{\textbf{j}}^{\epsilon}(x_{1,l_1},..,x_{m,l_m}, f(x_{1,l_1},...,x_{m,l_m}))\\
  &=v_{\textbf{j}}^{\epsilon}(x_{1,l_1},..,x_{m,l_m}, y_{l})=y_{\tau_1(j_1,l_1)...\tau_m(j_m,l_m)}=y_{\textbf{i}}.
 \end{align*}
Therefore, $T^{\epsilon}f\in \mathcal{G}$ and this shows that $T^{\epsilon}$ is a map from $\mathcal{G}$ to $\mathcal{G}$. 

Now, let $f,g \in \mathcal{G}$, $X:=(x_1,...,x_m) \in \prod\limits_{k=1}^m I_{k,j_k} $ and   
\[
\|\gamma \|_{\infty}:=\max\left\{\gamma_{\mathbf{j}}: {\mathbf{j}}\in \prod\limits_{k=1}^m \mathbb{N}_k \right\}.
\]
Using  \eqref{2eq7} and \eqref{2eq10}, we establish the contractivity of $T$ as follows:
 
\begin{align*}
|(T^{\epsilon}f &-T^{\epsilon}g)(X)| =\\
&\left|v_{\mathbf{j}}^{\epsilon}((u_{1,j_1}^{\epsilon^1})^{-1}(x_{1}),..,(u_{m,j_m}^{\epsilon^m})^{-1}(x_{m}),f((u_{1,j_1}^{\epsilon^1})^{-1}(x_{1}),(u_{m,j_m}^{\epsilon^m})^{-1}(x_{m})))\right.\\
& - \left. v_{\mathbf{j}}^{\epsilon}((u_{1,j_1}^{\epsilon^1})^{-1}(x_{1}),..,(u_{m,j_m}^{\epsilon^m})^{-1}(x_{m}), g((u_{1,j_1}^{\epsilon^1})^{-1}(x_{1}),...,(u_{m,j_m}^{\epsilon^m})^{-1}(x_{m})))\right|\\
         & \leq \gamma_{\mathbf{j}} \abs{ f((u_{1,j_1}^{\epsilon^1})^{-1}(x_{1}),...,(u_{m,j_m}^{\epsilon^m})^{-1}(x_{m})) -g((u_{1,j_1}^{\epsilon^1})^{-1}(x_{1}),...,(u_{m,j_m}^{\epsilon^m})^{-1}(x_{m})) }
         \\& \leq \| \gamma\, \|_{\infty}\, \| f-g \|_{\infty}. 
 \end{align*} 
As $X \in \prod\limits_{k=1}^m I_{k,j_k} $ was arbitrary, 
\[
\|T^{\epsilon}f-T^{\epsilon}g \|_{\infty} \leq \| \gamma \|_{\infty} \| f-g \|_{\infty}. 
\]
Using the Banach fixed point theorem, we conclude that $T^{\epsilon}$ has a unique fixed point $f^{\epsilon}$ in the complete metric spaces $\mathcal{G}$, i.e.,
$T^{\epsilon} f^{\epsilon}=f^{\epsilon}$. Equivalently, 
\begin{align}\label{2eq11}
f^{\epsilon} (x_1,\dots,x_m) &= \sum_{\mathbf{j}\in \prod\limits_{k=1}^{m}\mathbb{N}_{N_k}} v_{\mathbf{j}}^{\epsilon}((u_{1,j_1}^{\epsilon^1})^{-1}(x_1),\dots,(u_{m,j_m}^{\epsilon^m})^{-1}(x_m),\nonumber\\
  &\qquad f^{\epsilon}((u_{1,j_1}^{\epsilon^1})^{-1}(x_1),\dots,(u_{m,j_m}^{\epsilon^m})^{-1}(x_m))  \chi_{u_{\mathbf{j}}^{\epsilon}(\mathcal{I})}(x_1, \dots ,x_m),
\end{align}
for all $(x_1, \dots ,x_m)\in \mathcal{I}$.

Let us assume that for $X=(x_1,\dots,x_m)$, 
\[
(u_{\mathbf{j}}^{\epsilon})^{-1}(X)=((u_{1,j_1}^{\epsilon^1})^{-1}(x_1),\dots, (u_{m,j_m}^{\epsilon^m})^{-1}(x_m)) 
\]
and 
\[
u_{\mathbf{j}}^{\epsilon}(X)=(u_{1,j_1}^{\epsilon^1}(x_1), \dots ,u_{m,j_m}^{\epsilon^m}(x_m)). 
\]
Then, the self-referential equation associated with the multizipper FIF is given by
\begin{align}\label{2eq12}
f^{\epsilon}(X)= \sum_{\mathbf{j}\in \prod\limits_{k=1}^{m}\mathbb{N}_{N_k}} v_{\mathbf{j}}^{\epsilon}((u_{\mathbf{j}}^{\epsilon})^{-1}(X),f^{\epsilon}((u_{\mathbf{j}}^{\epsilon})^{-1}(X)))\chi_{u_{\mathbf{j}}^{\epsilon}(\mathcal{I})}(X), \quad\forall X\in \mathcal{I}.
\end{align}
The above equation can be rewritten as
\begin{align}\label{2eq13}
\begin{split}
 f^{\epsilon}(u_{\mathbf{j}}^{\epsilon}(X))=  v_{\mathbf{j}}^{\epsilon}(X,f^{\epsilon}(X)), \quad\forall X\in u_{\mathbf{j}}(\mathcal{I}).
    \end{split}
\end{align}
This unique fixed point $f^{\epsilon}$ interpolates the data points $\Delta$. For the    graph of $f^{\epsilon}$, $G=\{(X,f^{\epsilon}):X\in \mathcal{I}\}$, we obtain by\eqref{2eq8} and \eqref{2eq13},
\begin{align*}
\bigcup_{\mathbf{j}\in \prod\limits_{k=1}^{m}\mathbb{N}_{N_k}}W_{\mathbf{j}}^{\epsilon}(G)
        & =\bigcup_{{\mathbf{j}}\in \prod\limits_{k=1}^{m}\mathbb{N}_{N_k}} \{W_{\mathbf{j}}^{\epsilon}(X,f^{\epsilon}(X)): X\in \mathcal{I}\}\\
        & = \bigcup_{{\mathbf{j}}\in \prod\limits_{k=1}^{m}\mathbb{N}_{N_k}} \{(u_{\mathbf{j}}^{\epsilon}(X),v_{\mathbf{j}}^{\epsilon}(X,f^{\epsilon}(X))): X\in \mathcal{I}\}\\
        &=\bigcup_{\mathbf{j}\in \prod\limits_{k=1}^{m}\mathbb{N}_{N_k}}\{(u_{\mathbf{j}}^{\epsilon}(X),f^{\epsilon}(u_{\mathbf{j}}^{\epsilon}(X)): X\in \mathcal{I}\}\\
        &=\{(X,f^{\epsilon}(X)):X\in \mathcal{I}\}=G.
\end{align*}

\noindent The unique fixed point $f^{\epsilon}$ of  $T^{\epsilon}$ is called a {multivariate zipper FIF} corresponding to the IFS \eqref{2eq9}.
\end{proof}

\begin{remark} Note that in the construction of a multivariate zipper FIF,
we have to assign $\epsilon^k$ is either a zero or one column matrix for
$k=1,2,\dots,m.$ Then, we can obtain $2^m$- multivariate FIFs  by zipper methodology for the same set of scalings. When all $\epsilon = \bf{0}$, then the  multivariate zipper fractal function reduces to a simple multivariate fractal function \cite{Pandey_vishwanathan}.
\end{remark}
\subsection{Multivariate Zipper $\alpha$-Fractal Functions}\label{2subsec2.2}

For a given multivariate function $f\in C(\mathcal{I})$, consider a grid
\[
\Delta := \left\{(x_{1,j_1},\dots ,x_{m,j_m})\in \cI: {\mathbf{j}}\in \prod\limits_{k=1}^{m}\mathbb{N}_{N_k,0}\right\},
\]
on its domain where $a_k :=x_{k,0}< \dots <x_{k,N_k}=:b_k$ for each $k\in \mathbb{N}_{m}$.  First, we construct a continuous function $b:\mathcal{I} \to \mathbb{R}$ satisfying the conditions
\begin{equation}\label{2eq14}
  b(x_{1,j_1}, \dots ,x_{m,j_m})=f(x_{1,j_1},\dots ,x_{m,j_m}),\; \quad\forall \; \j\in \prod\limits_{k=1}^{m}\partial\mathbb{N}_{N_k,0},
\end{equation} 

For $ k\in \mathbb{N}_{m}$,  we define affine  maps $u_{k,j_k}^{\epsilon^k}: I_k\to I_{k,j_k}$ by
\begin{equation}\label{2eq15}
   u_{k,j_k}^{\epsilon^k}(x) :=a_{k,j_k}(x)+b_{k,j_k}, \quad j_k\in \mathbb{N}_{N_k},
\end{equation}
where $a_{k,j_k}$ and $b_{k,j_k}$ are chosen so that each map $u_{k,j_k}$ satisfies  \eqref{2eq1} and \eqref{2eq2}.

For ${\mathbf{j}}\in \prod\limits_{k=1}^{m}\mathbb{N}_{N_k}$, further define continuous variable scaling functions   
\begin{align}\label{2eq015}
  \alpha_{\mathbf{j}}: \mathcal{I} \to \mathbb{R}  
\end{align}
satisfying
\begin{itemize}
    \item[(i)] $\| \alpha_{\mathbf{j}}\|_{\infty}<1$,
    \item[(ii)] for all $j_k \in\intr\mathbb{N}_{N_k,0}$ and $(u_{k,j_k}^{\epsilon^k})^{-1}(x_{k,j_k}) = (u_{k,j_k+1}^{\epsilon^k})^{-1}(x_{k,j_k})=x_k^{*}$, $ (x_1, \dots ,x_m)\in \mathcal{I}$, 
\begin{align*}
\begin{split}
   \alpha_{j_1\cdots j_k\cdots j_m}(x_1,\cdots,x_{k-1},x_{k}^*,x_{k+1},\cdots,x_m, y)   =\alpha_{j_1\cdots j_k+1\cdots j_m}(x_1,\cdots,x_{k-1},\\
    x_{k}^*,x_{k+1},\cdots,x_m, y),
    (x_1,\dots,x_{k-1},x_{k+1},\dots ,x_m)\in  \prod\limits_{i=1,i\ne k}^{m}I_i,y \in \mathbb{R}.
\end{split}
\end{align*}
\end{itemize}
Further, define  $v_{\mathbf{j}}^{\epsilon}:\mathcal{I} \to \mathbb{R}$ by 
\begin{align}\label{2eq16}
    v_{\mathbf{j}}^{\epsilon}(X,y) :=f\left(u_{1,j_1}^{\epsilon^1}(x_1),\cdots,u_{m,j_m}^{\epsilon^m}(x_m)\right)+ \alpha_{\mathbf{j}}(X)(y-b(X)).    
\end{align}
Then, for all  ${\mathbf{j}}\in \prod\limits_{k=1}^{m}\mathbb{N}_{N_k}$, ${\mathbf{l}}=(l_1,\cdots,l_m)\in \prod\limits_{k=1}^{m}\partial \mathbb{N}_{N_k}$, we get 
\begin{align*}
v_{\mathbf{j}}^{\epsilon}(x_{1,l_1},\cdots,x_{m,l_m}, f(x_{1,l_1},\cdots,x_{m,l_m})) &=f(u_{1,j_1}^{\epsilon^1}(x_{1,l_1}),\cdots ,u_{m,j_m}^{\epsilon^m}(x_{m,l_m}))\\
       &=f(x_{1,\tau_1(j_1,l_1)},\cdots,x_{m,\tau_m(j_m,l_m)})\\
        &=y_{\tau_1(j_1,l_1)\cdots \tau_m(j_m,l_m)}).
\end{align*}
In other words, $ v_{\mathbf{j}}^{\epsilon}$ satisfies \eqref{2eq6}.

Now, suppose that $j_k \in\intr\mathbb{N}_{N_k,0}$, $1\leq k \leq m $, and that
\[
x_k^*:=(u_{k,j_k}^{\epsilon^k})^{-1}(x_{k,j_k})=(u_{k,j_k+1}^{\epsilon^k})^{-1}(x_{k,j_k}).
\]
For any $y\in \mathbb{R}$,
\begin{align*}
& v_{j_1,\ldots,j_{k-1},  j_k, j_{k+1},\ldots, j_m}^\epsilon(x_1,\ldots,x_{k-1},x_k^*,x_{k+1},\ldots,x_m,y) \\
& \qquad = f(u_{1,j_1}^{\epsilon^1}(x_1),\ldots,u_{k-1,j_{k-1}}^{\epsilon^{k-1}}(x_{k-1}),u_{k,j_k}^{\epsilon^k}(x_k^*),u_{k+1,j_{k+1}}^{\epsilon^{k+1}}(x_{k+1}),\ldots,u_{m,j_m}^{\epsilon^m}(x_m))\\
& \qquad\quad +\alpha_{j_1,\cdots,j_k,\cdots, j_m}(x_1,\cdots,x_k^*,\cdots,x_m)(y-b(x_1,\ldots,x_m))\\
& \qquad= f(u_{1,j_1}^{\epsilon^1}(x_1),\ldots,u_{k-1,j_{k-1}}^{\epsilon^{k-1}}(x_{k-1}),x_{k,j_k},u_{k+1,j_{k+1}}^{\epsilon^{k+1}}(x_{k+1}),\ldots,u_{m,j_m}^{\epsilon^m}(x_m))\\
& \qquad\quad + \alpha_{j_1,\cdots ,j_m}(x_1,\cdots,x_k^*,\cdots,x_m)(y-b(x_1,\ldots,x_m))\\
& \qquad =f(u_{1,j_1}^{\epsilon^1}(x_1),\ldots,u_{k-1,j_{k-1}}^{\epsilon^{k-1}}(x_{k-1}),u_{k,j_k+1}^{\epsilon_k}(x_{k,j_k}),u_{k+1,j_{k+1}}^{\epsilon^{k+1}}(x_{k+1}),\ldots,u_{m,j_m}^{\epsilon^m}(x_m))\\
& \qquad\quad + \alpha_{j_1,\cdots,j_k+1,\cdots, j_m}(x_1,\cdots,x_k^*,\cdots,x_m)(y-b(x_1,\ldots,x_m))\\
& \qquad =v_{j_1,\ldots,j_{k-1}, j_k+1, j_{k+1},\ldots,j_m}^{\epsilon} (x_1,\ldots,x_{k-1},x_k^*,x_{k+1},\ldots,x_m,y)).
\end{align*}
Therefore,  $v_{\mathbf{j}}^{\epsilon}$ satisfies \eqref{2eq0}, \eqref{2eq6}, and \eqref{2eq7}, for all ${\mathbf{j}}\in \prod\limits_{k=1}^{m}\mathbb{N}_{N_k}$.
Theorem \ref{2eq0} now implies that the 
 IFS 
 \begin{align}\label{2eq000}
 I^{\epsilon}=\left\{ \mathcal{K},W_{\mathbf{j}}^{\epsilon}:\mathbf{j}\in \prod\limits_{k=1}^{m}\mathbb{N}_{N_k}\right\}
 \end{align}
 defined in \eqref{2eq9}, where the maps $u_{k,j_k}^{\epsilon_k}$ and $v_{\mathbf{j}}^{\epsilon}$ are now defined as in \eqref{2eq15} and \eqref{2eq16}, determines a fractal function referred to as a {multivariate zipper $\alpha$-fractal function} and denoted by $f_{\Delta,b}^{\alpha,\epsilon}$.
 
The fractal function $f_{\Delta,b}^{\alpha,\epsilon}$ is the fixed point of the RB operator $ T^{\epsilon}:\mathcal{G}\to \mathcal{G} $ given by 
\begin{align}\label{2eq17}
T^{\epsilon}g(X) = f(X)+ \sum_{\j\in \prod\limits_{k=1}^{m}\mathbb{N}_{N_k}} \alpha_{\j} ((u_{\mathbf{j}}^{\epsilon})^{-1}(X))(f-b)((u_{\mathbf{j}}^{\epsilon})^{-1}(X)) \chi_{u_{\j}^{\epsilon}(\mathcal{I})}(X),
\forall X\in \mathcal{I}.
\end{align}
The fixed point $f_{\Delta,b}^{\alpha,\epsilon}$ satisfies the self-referential equation
\begin{align}\label{2eq18}
\begin{split}
   f_{\Delta,b}^{\alpha,\epsilon}(X) = f(X)+  \sum_{\mathbf{j}\in \prod\limits_{k=1}^{m}\mathbb{N}_{N_k}} \alpha_{\mathbf{j}} ((u_{\mathbf{j}}^{\epsilon})^{-1}(X))(f_{\Delta,b}^{\alpha,\epsilon}-b) ((u_{\mathbf{j}}^{\epsilon})^{-1}(X)) \chi_{u_{\mathbf{j}}^{\epsilon}(\mathcal{I})}(X),\\
   \nonumber \forall X\in \mathcal{I}.
\end{split}
\end{align}
or, equivalently,
\begin{equation}\label{2eq19}
f_{\Delta,b}^{\alpha,\epsilon}(u_{\mathbf{j}}^{\epsilon}(X))=f(u_{\mathbf{j}}^{\epsilon}(X))+\alpha_{\mathbf{j}}(X)(f_{\Delta,b}^{\alpha,\epsilon}(X)-b(X)),
\end{equation}
for all $X\in \prod\limits_{k=1}^{m}I_{k,j_k}$, $\mathbf{j}\in \prod\limits_{k=1}^{m}\mathbb{N}_{N_k}$.


We can easily establish the following inequality from \eqref{2eq19}:
\begin{align}\label{2eq21}
\|f_{\Delta,b}^{\alpha,\epsilon}-f\|_{\infty}\leq \dfrac{\|\alpha\|_{\infty}}{1-\|\alpha\|_{\infty}} \|f-b\|_{\infty},
\end{align}
 where $\|\alpha\|_{\infty}:=\max\left\{\|\alpha_{\j}\|_{\infty}: \j\in \prod\limits_{k=1}^{m}\mathbb{N}_{N_k}\right\}$.
 
 \noindent From \eqref{2eq21}, we observe that $\|f_{\Delta,b}^{\alpha,\epsilon}-f\|_{\infty} \to 0$ as $\|\alpha\|_{\infty} \to 0$.

\section{Multivariate Bernstein Zipper Fractal Function}\label{2sec3}

To obtain convergence of a multivariate $\alpha$-fractal function $f_{\Delta,b}^{\alpha,\epsilon}$ to $f$ without altering the scaling function $\alpha$, we take as base functions $b$ multivariate Bernstein polynomials $B_{\n}f(X)$\cite{Foupouagnigni_Wouodjie,Davis} of $f$.
The $\n:=(n_1,\dots,n_m)$-th Bernstein polynomial for  $f\in C(\mathcal{I})$ is given by
\begin{align}\label{2eq22}
 B_{\n}f(X) &=  \sum _{k_1=0}^{n_1}\cdots\sum _{k_m=0}^{n_m}  f\left(x_{1,0} +(x_{1,N_1}-x_{1,0}) \dfrac{k_1}{n_1},\dots,x_{m,0}\right.\nonumber \\ 
 & \qquad + (x_{m,N_m}-x_{m,0}) \left.\dfrac{k_m}{n_m}\right) \prod\limits _{r=1}^{m} b_{k_r,n_r}(x_r) ,
\end{align} 
where 
\begin{align*}
b_{k_r,n_r}(x_r) &=\binom{n_r}{k_r}\frac{(x_r-x_{r,0})^{k_r}(x_{r,N_r}-x_r)^{n_r-k_r}}{(x_{r,N_r}-x_{r,0})^{n_r}},\quad 0\leq k_r \leq n_r,\\
\end{align*}
for $r=1,...,m$ and $n_1,...,n_m\in \mathbb{N}$.

If we use as base function $b(X):= B_{\n}f(X)$ in \eqref{2eq16}, then  the IFS \eqref{2eq000} becomes
\begin{align}\label{2eq23}
  I_{\n}^{\epsilon}=\left\{ K,W_{\j}^{\epsilon}: \j\in \prod\limits_{k=1}^{m}\mathbb{N}_{N_k}\right\}, 
\end{align}
where
\[
W_{\mathbf{j}}^{\epsilon}(X,y) :=\left(u_{1,j_1}^{\epsilon^1}(x_1),\dots,u_{m,j_m}^{\epsilon^m}(x_m), v_{\mathbf{j}}^{\epsilon}(X,y)\right), 
\]
and
\[
v_{\mathbf{j}}^{\epsilon}(X,y) :=f\left(u_{1,j_1}^{\epsilon^1}(x_1),\dots,u_{m,j_m}^{\epsilon^m}(x_m)\right)+\alpha_{\mathbf{j}}(X)(y-B_{\n}f(X),
\]
for all $\mathbf{j}\in \prod\limits_{k=1}^{m}\mathbb{N}_{N_k}$. This IFS determines 
a  multivariate zipper $\alpha$-fractal function 
\[
f_{\Delta,B_{\n}}^{\alpha,\epsilon} :=f_{\Delta;\n}^{\alpha,\epsilon}:=f_{\n}^{\alpha,\epsilon} 
\]
(we use these three notations interchangeably) referred to as a {multivariate  Bernstein zipper $\alpha$-fractal function} corresponding to the continuous function $f:\mathcal{I} \to \mathbb{R}$. It satisfies the self-referential equation
\begin{align}\label{2eq24}
f_{\Delta;\n}^{\alpha,\epsilon}\circ u_{\mathbf{j}}^{\epsilon} = f\circ u_{\mathbf{j}}^{\epsilon} + \alpha_{\mathbf{j}}\left(f_{\Delta;\n}^{\alpha,\epsilon} - B_{\n}f\right),\text{ on $X$ and for all $\mathbf{j}\in \prod\limits_{k=1}^{m}\mathbb{N}_{N_k}$}.
\end{align}

\begin{definition}
Define an operator $\mathcal{F}_{\Delta,B_{\n}}^{\alpha,\epsilon}:C(\mathcal{I})\to C(\mathcal{I})$ by 
\begin{center}
    $\mathcal{F}_{\Delta,B_{\n}}^{\alpha,\epsilon}(f) :=f_{\Delta,B_{\n}}^{\alpha,\epsilon}=f_{\Delta;\n}^{\alpha,\epsilon}$,
\end{center}
where $\Delta$ is the set of data points, $B_{\n}$ a multivariate  Bernstein operator and $\alpha$ is scaling function. We call this operator a multivariate  Bernstein zipper $\alpha$-fractal operator.
\end{definition}

\begin{theorem}\label{2th3.1}
The multivariate Bernstein zipper $\alpha$-fractal operator
 \[
 \mathcal{F}_{\Delta,B_{\n}}^{\alpha,\epsilon}:C(\mathcal{I})\to C(\mathcal{I})
 \]
 is linear and bounded.
\end{theorem}
\begin{proof}
The proof of this theorem is the same as in the univariate case for the $\alpha$-fractal operator in \cite{Chand_Navascues}.\end{proof}
Without altering $\alpha$, we otbain the following convergence result.
\begin{theorem}\label{2th3.2}
Let $f\in C(\mathcal{I})$. Then the multivariate Bernstein zipper $\alpha$-fractal function $f_{\Delta,B_{\n}}^{\alpha,\epsilon}$ converges uniformly to $f$ as $n_i \to \infty $, for all $1\leq i \leq m$.
\end{theorem}
\begin{proof}
From  \eqref{2eq24}, we get 
\begin{align*}
\|f_{\Delta,B_{\n}}^{\alpha,\epsilon}-f\|_{\infty} &\leq  \|\alpha\|_{\infty}  \|f_{\Delta,B_{\n}}^{\alpha,\epsilon}-B_{\n}f\|_{\infty}\\
   & \leq \|\alpha\|_{\infty}\|f_{\Delta,B_{\n}}^{\alpha,\epsilon}-f\|_{\infty}+\|\alpha\|_{\infty}\|f-B_{\n}f\|_{\infty}
\end{align*}
Hence,
\begin{align}\label{2eq25}
     \|f_{\Delta,B_{\n}}^{\alpha,\epsilon}-f\|_{\infty}\leq \dfrac{\|\alpha\|_{\infty}}{1-\|\alpha\|_{\infty}} \|f-B_{\n}f\|_{\infty}.
\end{align}
By Ref. \cite{Foupouagnigni_Wouodjie}, we know that  $\|f-B_{\n}f\|_{\infty} \to 0 $ as $n_i \to \infty $, for all $1\leq i \leq m$. Employing this in \eqref{2eq25}, we obtain $\|f_{\Delta,B_{\n}}^{\alpha,\epsilon}-f\|_{\infty} \to 0$, as $n_i \to \infty $ for all $1\leq i \leq m$. Therefore, $f_{\Delta,B_{\n}}^{\alpha,\epsilon}$ converges uniformly to $f$ as  $n_i \to \infty $ for all $1\leq i \leq m$.
\end{proof}
\begin{example}
In this example,  we  provide an illustration of Theorem \ref{2th3.2}. Let  $f(x) :=\sin(\frac{\pi}{2}xy)$ in $ \mathcal{I} :=I_1\times I_2$ where $ I_1 =I_2 :=[0,1]$, $\alpha_{j_1, j_2}=0.5$ for all $ (j_1,j_2) \in \mathbb{N}_{3}\times  \mathbb{N}_{3}$. Consider the grid on $\R^2$ given by 
\begin{align*}
    \Delta :=\left\{ (x_i, y_j)\in \R^2 : x_i ,  y_j \in  \{0, \tfrac{1}{3}, \tfrac{2}{2}, 1\}\right\}.
    \end{align*}
The original bivariate function  $f(x)=\sin(\frac{\pi}{2}xy)$  is constructed in 
Fig. \ref{2fig1}(a). For the interpolation data of $f$ on $\Delta$, we 
have constructed fractal functions 
in Figs.  \ref{2fig1}(b)-(e) corresponding to different values of the signature.
    Fig. \ref{2fig1}(f) is the plot of  $f_{\Delta,B_{20,20}}^{\alpha,\epsilon}$
    with binary signature matrix 1.   One can observe from Figs. \ref{2fig1}(d) and \ref{2fig1}(f) that $f_{\Delta,B_{20,20}}^{\alpha,\epsilon}$ provides a better approximation for $f\in C(\cI)$ than that by $f_{\Delta,B_{3,3}}^{\alpha,\epsilon}$.
\end{example}

\begin{figure}
     \centering
     \begin{subfigure}{0.49\textwidth}
         \centering
         \includegraphics[width=\textwidth]{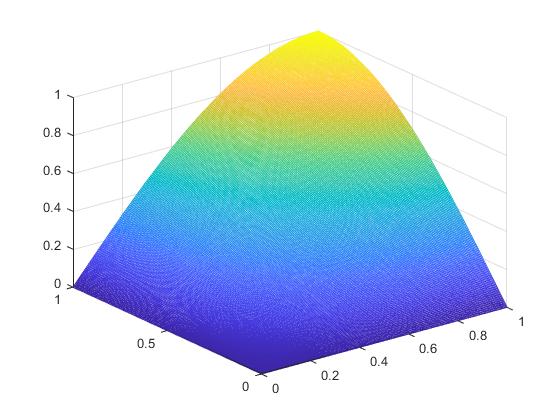}
         \caption{$f(x_1,x_2)=\sin(\frac{\pi}{2}x_1x_2)$}
         \label{2fig1a}
     \end{subfigure}
     \hfill
     \begin{subfigure}{0.49\textwidth}
         \centering
         \includegraphics[width=\textwidth]{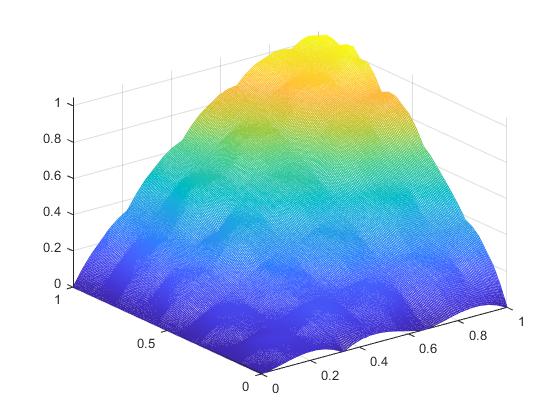}
         \caption{$f_{\Delta,B_{3,3}}^{\alpha,\epsilon} \;\text{when}\;\epsilon=(0,1)$}
         \label{2fig1b}
     \end{subfigure}
     \hfill
      \begin{subfigure}{0.49\textwidth}
         \centering
         \includegraphics[width=\textwidth]{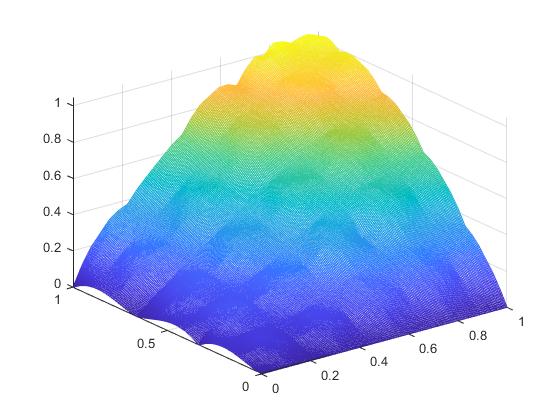}
         \caption{$f_{\Delta,B_{3,3}}^{\alpha,\epsilon}\;\text{when}\;\epsilon=(1,0)$}
         \label{2fig1c}
     \end{subfigure}
     \hfill
     \begin{subfigure}{0.49\textwidth}
         \centering
         \includegraphics[width=\textwidth]{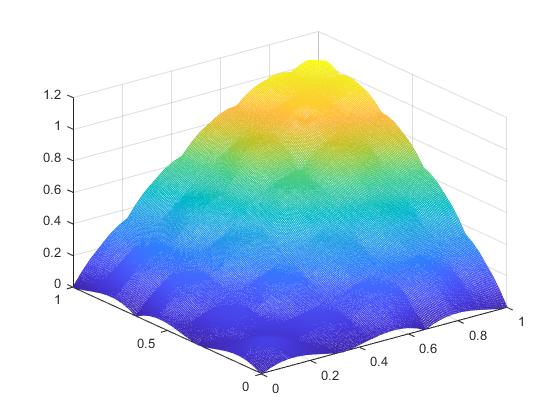}
         \caption{$f_{\Delta,B_{3,3}}^{\alpha,\epsilon} \;\text{when}\;\epsilon=(1,1)$}
         \label{2fig1d}
     \end{subfigure}
     \hfill
     \begin{subfigure}{0.49\textwidth}
         \centering
         \includegraphics[width=\textwidth]{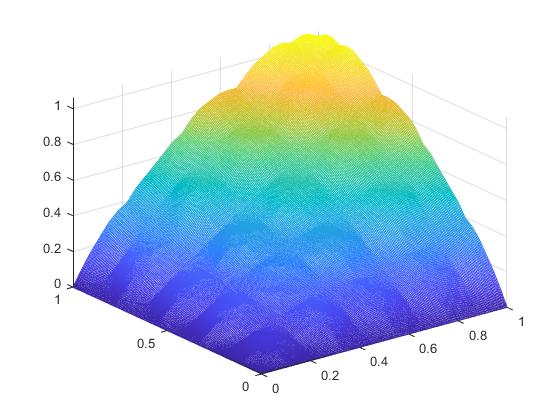}
         \caption{$f_{\Delta,B_{3,3}}^{\alpha,\epsilon}\;\text{when}\;\epsilon=(0,0)$}
         \label{2fig1e}
     \end{subfigure}
     \hfill
     \begin{subfigure}{0.49\textwidth}
         \centering
         \includegraphics[width=\textwidth]{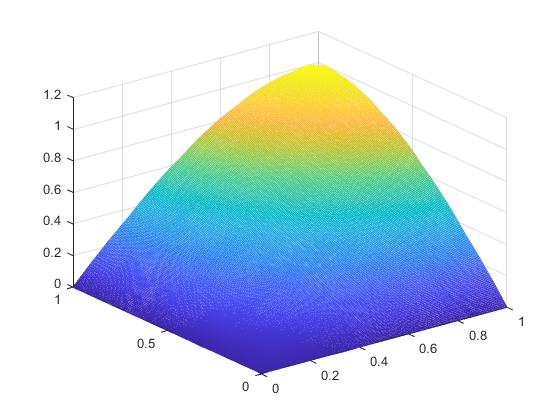}
         \caption{$f_{\Delta,B_{20,20}}^{\alpha,\epsilon} \;\text{when}\;\epsilon=(1,1)$}
         \label{2fig1f}
     \end{subfigure}
     \hfill
        \caption{Multivariate Bernstein zipper $\alpha$-fractal functions}
        \label{2fig1}
\end{figure}
 \section{Constrained Multivariate Bernstein zipper $\alpha$-Fractal Approximation}\label{2sec4}
 In this section, we study constrained approximation by multivariate Bernstein zipper $\alpha$-fractal functions.
 \begin{theorem}\label{2th3}
 Let $f\in C(\mathcal{I})$ and suppose $f(X)\geq 0$ for all $X \in \mathcal{I}$. Consider the set
 \[
 \Delta := \left\{(x_{1,j_1},\dots ,x_{m,j_m}): \mathbf{j} \in \prod\limits_{k=1}^{m}\mathbb{N}_{N_k,0}\right\} 
 \]
where $a_k:=x_{k,0}<\dots<x_{k,N_k}=:b_k$ for each $k\in\N_m$,  $I_k :=[a_k,b_k]$, and $\alpha:\mathcal{I}\to \mathbb{R}$ is a continuous scaling function. Then,  the sequence $\{I_{\n}^{\epsilon}\}$ of IFSs \eqref{2eq23} determines a sequence $\{f^{\alpha,\epsilon}_{\Delta;\n}\}$ of positive multivariate Bernstein zipper $\alpha$-fractal functions that converges uniformly to $f$ if the scaling functions $\alpha_{\mathbf{j}}(X)$ are chosen as in \eqref{2eq015} and according to
\begin{align}\label{2eq26}
\max\left\{ \frac{-\phi^{\epsilon}(f;\mathbf{j})}{C_{\n}-\phi_{\n}},-\frac{C_{\n}-\Phi^{\epsilon}(f;\mathbf{j})}{\Phi_{\n}}\right\}\leq \alpha_{\j}(X) \leq 
     \min\left\{\frac{\phi^{\epsilon}(f;\mathbf{j})}{\Phi_{\n}},\frac{C_{\n}-\Phi^{\epsilon}(f;\mathbf{j})}{C_{\n}-\phi_{\n}}\right\},
 \end{align}
for $\mathbf{j}\in \prod\limits_{k=1}^{m}\mathbb{N}_{N_k}$, where 
\begin{gather*}
\phi^{\epsilon}(f;\mathbf{j}) :=\min_{X\in I}f(u_{\mathbf{j}}^{\epsilon}(X)),
 \quad \Phi^{\epsilon}(f;\mathbf{j}) :=\max_{X\in I}f(u_{\mathbf{j}}^{\epsilon}(X)),\\
 \phi_{\n} :=\min_{X\in I}B_{\n}f(X),\quad \Phi_{\n} :=\max_{X\in I}B_{\n}f(X), 
 \end{gather*}
 and $C_{\n}$ is a positive real number strictly greater than both $\phi_{\n}$ and $\|f\|_{\infty}$.
 \end{theorem} 
\begin{proof}
By Theorem \ref{2th3.2}, there exists a sequence $\{f^{\alpha,\epsilon}_{\Delta;\n}\}$, for $n_k \in \mathbb{N} $, of multivariate   Bernstein zipper $\alpha$-fractal functions that converges to $f$ for any given non-negative function
    $f\in C(\mathcal{I})$. By \cite{Foupouagnigni_Wouodjie}, $B_{\n}$ is a positive linear operator and thus $B_{\n}f(X)\geq 0$, for all $X\in \mathcal{I}$, which implies the positivity of  $\Phi_{\n}$.
     
Let $q_{\n;\mathbf{j}}^{\epsilon}(X) :=f(u_{\mathbf{j}}^{\epsilon}(X))-\alpha_{\mathbf{j}}(X)B_{\n}f(X)$.
    By \eqref{2eq24}, we obtain
\begin{align}\label{2eq27}
f^{\alpha,\epsilon}_{\Delta;\n}(u_{\mathbf{j}}^{\epsilon}(X))
        & = f(u_{\mathbf{j}}(X))+\alpha_{\mathbf{j}}(X)(f^{\alpha,\epsilon}_{\textbf{n}}(X)-B_{\n}f(X))\nonumber\\
         & = v_{\n;\mathbf{j}}^{\epsilon}(X,f^{\alpha,\epsilon}_{\Delta;\n}(X)).
\end{align}
As $v_{\n;\mathbf{j}}^{\epsilon}(X,y) \in [0,C_{\n}]$, $\mathbf{j}\in \prod\limits_{k=1}^{m}\mathbb{N}_{N_k}$, for all $(X,y)\in \mathcal{I} \times[0,C_{\n}]$ this implies 
\[
f^{\alpha,\epsilon}_{\Delta;\n}(u_{\mathbf{j}}^{\epsilon}(X)\in[0,C_{\n}], \quad\forall X\in \mathcal{I}.
\]
Therefore, in order to prove that $f^{\alpha}_{\Delta;\n}(X)\in[0,C_{\n}]$, for all $X\in \mathcal{I}$, it suffices to show that $v_{\n;\mathbf{j}}^{\epsilon}(X,y) \in [0,C_{\n}],$ for all $(X,y)\in \mathcal{I}\times[0,C_{\n}]$. 

Suppose that $(X,y)\in \mathcal{I}\times[0,C_{\n}]$ and  $|\alpha_{\mathbf{j}}(X)|<1$. Now there are two cases: 
\vskip 4pt
\noindent
\textbf{Case (i)}: Let $ 0\leq \alpha_{\mathbf{j}}(X)<1$, for all $X\in \mathcal{I}$. 

Then, $ 0\leq y\leq C_{\n}$ gives 
\[
q_{\n;\mathbf{j}}^{\epsilon}(X)\leq \alpha_{\mathbf{j}}(X)y+q_{\n;\mathbf{j}}^{\epsilon}(X)\leq C_{\n}\alpha_{\mathbf{j}}(X)+q_{\n;\mathbf{j}}^{\epsilon}(X).
\]
 Hence, for $ \mathbf{j}\in \mathcal{I}$ and $(X,y)\in \mathcal{I}\times[0,C_{\n}]$,
\begin{equation*}
    0\leq v_{\n;\mathbf{j}}^{\epsilon}(X,y) \leq C_{\n}   
\end{equation*}
   holds if 
\begin{equation}\label{2eq28}
\begin{split}
& f(u_{\mathbf{j}}^{\epsilon}(X))-\alpha_{\mathbf{j}}(X) B_{\n}f(X) \geq 0,\\
& f(u_{\mathbf{j}}^{\epsilon}(X))-\alpha_{\mathbf{j}}(X)B_{\n}f(X) \leq C_{\n}(1-\alpha_{\mathbf{j}}(X)).
\end{split}
\end{equation}
As $f(u_{\mathbf{j}}^{\epsilon}(X))\geq \phi^{\epsilon}(f;\mathbf{j})$ and $B_{\n}f(X)\leq \Phi_{\n},$ we obtain that  
\[
f(u_{\mathbf{j}}^{\epsilon}(X)) -\alpha_{\mathbf{j}}(X)B_{\n}f(X)\geq 0 
\]
provided 
\begin{equation*}
\phi^{\epsilon}(f,\mathbf{j})-\alpha_{\mathbf{j}}(X)\Phi_{\n}\geq 0.
\end{equation*}
Hence $ \alpha_{\mathbf{j}}(X) \leq \frac{\phi^{\epsilon}(f,\j)}{\Phi_{\n}}.$ 

Next, as
$f(u_{\mathbf{j}}^{\epsilon}(X))\leq \Phi^{\epsilon}(f,\mathbf{j})$ and $B_{\n}f(X)\geq \phi_{\n}$, the second inequality in \eqref{2eq28} holds if 
\begin{equation*}
    \begin{split}
        \alpha_{\mathbf{j}}(X)\leq \frac{C_{\n}-\Phi^{\epsilon}(f,\mathbf{j})}{C_{\n}-\Phi_{\n}^{\epsilon}}.
    \end{split}
\end{equation*}
In this case,  for $ \mathbf{j}\in \prod\limits_{k=1}^{m}\mathbb{N}_{N_k}$ and $(X,y)\in \mathcal{I}\times[0,C_{\n}]$, $v_{\n; \mathbf{j}}^{\epsilon}(X,y) \in [0,C_{\n}]$   
is true whenever
\begin{equation*}
    \begin{split}
      \alpha_{\mathbf{j}}(X) \leq \min \left\{ \frac{\phi^{\epsilon}(f,\mathbf{j})}{\Phi_{\n}}, \frac{C_{\n}-\Phi^{\epsilon}(f,\mathbf{j})}{C_{\n}-\phi_{\n}}\right\}.  
    \end{split}
\end{equation*}
\textbf{Case (ii)}: Let $ -1<\alpha_{\mathbf{j}}(X)\leq 0$, for all $X\in I$. 

Then $ 0\leq y\leq C_{\n}$ implies 
\[
C_{\n}\alpha_\mathbf{j}(X)+q_{\n;\mathbf{\mathbf{j}}}^{\epsilon}\leq \alpha_{\mathbf{j}}(X)y+q_{\n;\mathbf{j}}^{\epsilon}\leq q_{\n;\mathbf{j}}^{\epsilon}.
\]
 For  $ \mathbf{j}\in \prod\limits_{k=1}^{m}\mathbb{N}_{N_k}$, $(X,y)\in \mathcal{I} \times[0,C_{\n}]$, the   inequality 
\begin{center}
   $0\leq v_{\n;\mathbf{j}}^{\epsilon}(X,y)=\alpha_{\mathbf{j}}(X)y+q_{\n;\mathbf{j}}^{\epsilon}\leq C_{\n},$ 
\end{center}
 holds if 
\begin{equation}\label{2eq29}
\begin{split}
& f(u_{\mathbf{j}}^{\epsilon}(X))-\alpha_{j}(X)B_{n}f(X)\leq C_{\n},\\
&C_{\n}\alpha_{\mathbf{j}}(X)+f(u_{\mathbf{j}}^{\epsilon}(X))-\alpha_{\mathbf{j}}(X)B_{n}f(X)\geq 0.
\end{split}
\end{equation}
As $f(u_{\mathbf{j}}^{\epsilon}(X))\leq \Phi^{\epsilon}(f,\mathbf{j})$ and $B_{\n}f(X)\leq \Phi_{\n}$,  then from first inequality in  \eqref{2eq29}, we obtain
\begin{equation*}
f(u_{\mathbf{j}}^{\epsilon}(X)-\alpha_{\mathbf{j}}(X)B_{\n}f(X)\leq \Phi^{\epsilon}(f,\mathbf{j})-\alpha_{\mathbf{j}}(X)\Phi_{\n}\leq C_{\n}.
\end{equation*}
The last part of the above inequality reduces to  $\alpha_{\mathbf{j}}(X) \geq -\frac{C_{\mathbf{n}}-\Phi^{\epsilon}(f,\mathbf{j})}{\Phi_{\n}}.$ 

Further, since $B_{\n}f(X)\geq \phi_{\n}$ {and}  $f(u_{\mathbf{j}}^{\epsilon}(X))\geq \phi^{\epsilon}(f,\mathbf{j})$, a simple calculation yields that the second inequality in \eqref{2eq29} holds if $\alpha_{\mathbf{j}}(X)\geq \frac{-\phi^{\epsilon}(f,\mathbf{j})}{C_{\n}-\phi_{\n}}$. In this case, for $\mathbf{\mathbf{j}}\in \prod\limits_{k=1}^{m}\mathbb{N}_{N_k}$ and  $(X,y)\in \mathcal{I}\times[0,C_{\n}]$, $v_{{\n};\mathbf{j}}^{\epsilon}(X,y) \in [0,C_{\n}]$ holds if
\begin{equation*}
\max\left\{ \frac{-\phi^{\epsilon}(f,\mathbf{j})}{C_{\n}-\phi_{\n}},-\frac{C_{\n}-\Phi^{\epsilon}(f,\mathbf{j})}{\Phi_{\n}}\right\}\leq \alpha_{\mathbf{j}}(X).
\end{equation*} 
These two cases imply \eqref{2eq26}.
\end{proof}

In the above theorem, we have seen that for every continuous function $f:\cI\to\R$ with $f\geq 0$ on $\cI$, there exists a sequence of positive  multivariate Bernstein zipper $\alpha$-fractal  functions which converges to $f$ in the sup-norm. The next result considers the case when the difference of two functions in $C(\cI)$ is positive.
\begin{theorem}\label{2th4}
Let $f,g\in C(\mathcal{I})$ and $f\geq g$ on $\mathcal{I}$. For all $\n \in \mathbb{N}^{m}$ let $f^{\alpha,\epsilon}_{\Delta;\n}$ be multivariate Bernstein $\alpha$-fractal functions associated with  the IFS $I_{\n}^{\epsilon}$, where
\[
\Delta := \left\{(x_{1,j_1},...,x_{m,j_m}): \mathbf{j}\in \prod\limits_{k=1}^{m}\mathbb{N}_{N_k,0}\right\}
\]
such that $a_k:=x_{k,0}<...<x_{k,N_k}=:b_k$, for $k\in\N_m$,  $I_k :=[a_k,b_k]$ and $\alpha_{\j}$ taken as in \eqref{2eq015}.

Then, the sequence $\{I_{\n}^{\epsilon}\}$ of IFSs determines a sequence  of  multivariate Bernstein zipper $\alpha$-fractal functions $\{f^{\alpha,{\epsilon}}_{\Delta;\n}\}$ such that $f^{\alpha,\epsilon}_{\Delta;\n}\geq g$ on $\mathcal{I}$ and which converges uniformly to $f$ if the continuous scaling functions $\alpha_{\mathbf{j}}(X)$ are chosen as in \eqref{2eq015} and satisfy

\begin{align}\label{2eq30}
    0\leq \alpha_{\mathbf{j}}(X)\leq \min\left\{\frac{\phi^{\epsilon}(f-g,\mathbf{j})}{\Phi_{\n}(f)-\phi(g)},1 \right\},
\end{align}
where $\phi^{\epsilon}(f-g,\mathbf{j}) :=\min\limits_{X\in \mathcal{I}}B_{n}f(X)$ and   $\phi(g):=\min\limits_{X\in \mathcal{I}}g(X).$
\end{theorem}
\begin{proof}
By \eqref{2eq24}, we can rewrite the functional equation of $f_{\Delta;\n}^{\alpha,\epsilon}$ as follows.
\begin{align}\label{2eq32}
f^{\alpha,\epsilon}_{\Delta;\n}(X) &= f(X)+\sum_{\mathbf{j}\in \prod\limits_{k=1}^m\mathbb{N}_k} \alpha_{\mathbf{j}} ((u_{\mathbf{j}}^{\epsilon})^{-1}(X))
   (f^{\alpha}_{\Delta;\n}((u_{\mathbf{j}}^{\epsilon})^{-1}(X))\nonumber \\
   & \qquad -B_{\n}f((u_{\mathbf{j}}^{\epsilon})^{-1}(X)))\chi_{u_{\mathbf{j}}^{\epsilon}(I)}(X), \quad X\in \mathcal{I}.
\end{align}
This functional equation is a rule to get the values of $f^{\alpha,\epsilon}_{\Delta;\n}$ at $(N^{r+2}+1)^m$ distinct points in $\mathcal{I}$ in $(r+1)$-th iteration using the value of $f^{\alpha,\epsilon}_{\Delta;\n}$ at $(N^{r+1}+1)^m$ points in $\mathcal{I}$ at the $r$-th iteration. 

Let us begin the iteration process with the nodal points $X_i,i\in \mathbb{N}$. We establish that the $p$-th iterated image of $X$ satisfies $f^{\alpha,\epsilon}_{\Delta;\n}(X)\geq g(X)$.
For the $0$-th iteration, we have 
\[
f^{\alpha,\epsilon}_{\Delta;\n}(X)\geq g(X),
\]
since $f^{\alpha,\epsilon}_{\Delta;\n}$ interpolates $f$ at the nodes and $f(X)\geq g(X)$.

Now, suppose that $f^{\alpha,\epsilon}_{\Delta;\n}\geq g$. We show that 
\begin{align*}
f^{\alpha,\epsilon}_{\Delta;\n}(u_{\mathbf{j}}^{\epsilon}(X))\geq g((u_{\mathbf{j}}^{\epsilon}(X)),\quad 
     \forall X\in \mathcal{I}, \;\forall \mathbf{j} \in \prod\limits_{k=1}^{m} \mathbb{N}_k.   
\end{align*} From the fixed point equation \eqref{2eq32}, this is equivalent to proving that
\begin{align}\label{2eq33}
     f(u_{\mathbf{j}}^{\epsilon}(X))+\alpha _{\mathbf{j}} (X)f^{\alpha,\epsilon}_{\Delta;\n}(X)-
     \alpha _{\mathbf{j}} (X)B_{\n}f(X)-g(u_{\mathbf{j}}^{\epsilon}(X))\geq 0.  
\end{align}
Choosing $\alpha _{\mathbf{j}}(X)$ as non-negative and using the $p$-th iterated image yields
\[
f(u_{\mathbf{j}}^{\epsilon}(X))+\alpha _{\mathbf{j}}(X)g(X)- \alpha_{\mathbf{j}} (X)B_{\n}f(X)-g(u_{\mathbf{j}}^{\epsilon}(X))\geq 0.
\]
For the validity of the above inequality, it suffices to choose   $\alpha _{\mathbf{j}}$ so that
\begin{align}\label{2eq34}
    0\leq \alpha_{\mathbf{j}}(X)\leq \min\left\{\frac{\phi^{\epsilon}(f-g,\mathbf{j})}{\Phi_{\n}(f)-\phi(g)} \right\}.
\end{align}
If $\alpha_{\mathbf{j}}$, $\mathbf{j}\in \prod\limits_{k=1}^{m} \mathbb{N}_{k}$ satisfies \eqref{2eq30}, then $f^{\alpha,\epsilon}_{\Delta;\n}\geq g$ on a dense subset of $\mathcal{I}$.
By a density and continuity argument, $f^{\alpha,\epsilon}_{\Delta;\n}(X)\geq g(X)$ for all $X\in \mathcal{I}$.
\end{proof} 

\begin{corollary}\label{2th5}
Let $f,g\in C(\mathcal{I})$ and $f\geq g$ on $\mathcal{I}$. Consider the partition  
\[
\Delta := \left\{(x_{1,j_1},\dots ,x_{m,j_m}): \j\in \prod\limits_{k=1}^{m}\mathbb{N}_{N_k,0}\right\}
\]
with $a_k :=x_{k,0}<\dots<x_{k,N_k}=:b_k$, for each $k\in \mathbb{N}_m$,  $I_k :=[a_k,b_k]$, and a continuous scaling function $\alpha_{\mathbf{j}}:\mathcal{I}\to \mathbb{R}$.

Then, there exist sequences  $\{f^{\alpha,\epsilon}_{\Delta;\n}\}$ and $\{g^{\alpha,\epsilon}_{\Delta;\n}\}$ of multivariate Bernstein zipper $\alpha$-fractal function converging to $f$ and $g$, respectively, with $f^{\alpha,\epsilon}_{\Delta;\n} \geq g^{\alpha,\epsilon}_{\Delta;\n}$ on $\mathcal{I}$, if the scaling functions satisfy \eqref{2eq015} as well as the following estimate:
\begin{align}\label{2eq35}
    0\leq \alpha_{\mathbf{j}}(X)\leq \min\left\{\frac{\phi^{\epsilon}(f-g,\mathbf{j})}{\Phi_{\n}(f-g)},1 \right\}, \quad \mathbf{j}\in \prod\limits_{k=1}^{m}\mathbb{N}_{N_k},
\end{align}
where 
\[
\phi^{\epsilon}(f-g,\mathbf{j}):=\min_{X\in I}(f-g)(u_{\mathbf{j}}^{\epsilon}(X))
\] 
and 
\[
\Phi_{\n}(f-g):=\max_{X\in \mathcal{I}}B_{\n}(f-g)(X).
\]
\end{corollary}
\begin{proof}
We obtain the result by taking $f$ as $f-g$ and $g=0$ in Theorem \ref{2th4}.
\end{proof}

In the following theorem, we  construct  a sequence of increasing  multivariate Bernstein zipper FIFs  and a one-sided approximation by a convex  continuous function defined in an $m$-dimensional hyperrectangle.  
For this theorem, we adopt the following notation: For $\n=(n_1,\cdots, n_m)\in \prod\limits_{k=1}^m I_k$, let $\n+1 :=(n_1+1,\cdots, n_m+1)$.
\begin{theorem}
Let $f\in C(\mathcal{I})$ be convex and suppose $\alpha_{\mathbf{j}}$ are non-negative scaling functions as in \eqref{2eq015}. Then, for $n_i \in \mathbb{N}_k$, $i\in \mathbb{N}_m$,
\begin{align}\label{2eq45}
 f^{\alpha,\epsilon}_{\Delta;\n}(X) \leq f^{\alpha,\epsilon}_{\Delta;\n+1}(X), \text{ for all } X\in   \mathcal{I}. 
\end{align}
Moreover, for $n_i \in \mathbb{N}_k$, $i\in \mathbb{N}_m$,
\begin{align}\label{2eq46}
 f^{\alpha,\epsilon}_{\Delta;\n}(X) \leq f(X), \text{ for all } X\in   \mathcal{I}. 
\end{align}
\end{theorem}
\begin{proof}
By \eqref{2eq24}, we have self-referential equations for $f^{\alpha,\epsilon}_{\Delta;\n}$ and  $f^{\alpha,\epsilon}_{\Delta;\n+1}$,  $\mathbf{j}\in \prod\limits_{k=1}^{m}\mathbb{N}_{N_k}$, $ X\in \mathcal{I}$, of the form
\begin{align}\label{2eq44}
\begin{split}
   f_{\Delta;\n}^{\alpha,\epsilon}(u_{\j}^{\epsilon}(X)) &=f(u_{\mathbf{j}}^{\epsilon}(X))+\alpha_{\j}(X)
   \cdot(f_{\Delta;\n}^{\alpha,\epsilon}(X)-B_{\n}f(X)),\\
   f_{\Delta;\n+1}^{\alpha,\epsilon}(u_{\j}^{\epsilon}(X)) &=f(u_{\j}^{\epsilon}(X))+\alpha_{\j}(X)(f_{\Delta;\n+1}^{\alpha,\epsilon}(X)-B_{\n+1}f(X))
\end{split}
\end{align}
From \eqref{2eq44}, we obtain
\begin{align*}
    \begin{split}
       f_{\Delta;\n+1}^{\alpha,\epsilon}(u_{\j}^{\epsilon}(X)) - f_{\Delta;\n}^{\alpha,\epsilon}(u_{\j}^{\epsilon}(X)) &=
       \alpha_{\j}(X)(f_{\Delta;\n+1}^{\alpha,\epsilon}(X)-f_{\Delta;\n}^{\alpha,\epsilon}(X))\\  
       & \quad + \alpha_{\j}(X)(B_{\n}f-B_{\n+1}f)(X).
    \end{split}
\end{align*}
Reference \cite[Theorem 5]{Foupouagnigni_Wouodjie} implies that $(B_{\n}f-B_{\n+1}f)(X)\geq 0$ and the above equation thus takes the form
\begin{align*}
       f_{\Delta;\n+1}^{\alpha,\epsilon}(u_{\j}^{\epsilon}(X)) - f_{\Delta;\n}^{\alpha,\epsilon}(u_{\j}^{\epsilon}(X))\leq 
       \alpha_{\j}(X)(f_{\Delta;\n+1}^{\alpha,\epsilon}(X)-f_{\Delta;\n}^{\alpha,\epsilon}(X)).
\end{align*}
As the construction of fractal function is an iterative process,  we infer from the above equation that  $f_{\Delta;\n+1}^{\alpha,\epsilon}(X)\geq f_{\Delta;\n}^{\alpha,\epsilon}(X)$, for all $ X\in \mathcal{I}$.

As $f_{\Delta;\n}^{\alpha,\epsilon}$ converges uniformly to $f$, \eqref{2eq45} implies \eqref{2eq46}.
\end{proof}

\section{Coordinate-Wise Monotonic Multivariate Bernstein zipper $\alpha$-fractal functions}\label{2sec6}
Multivariate monotonic interpolation functions play an important role in, for instance, 
empirical option pricing models \cite{Hutchison_Lo_Poggio} in finance, 
design of aggregation operators in multi-criteria decision-making and fuzzy logic\cite{Calvo_Kolesarova_Komornikova_Mesiar1}, dose-response curves and surfaces in biochemistry and pharmacology. Some work on monotonic surface
approximation can be found in \cite{Chand_vv, Calvo_Kolesarova_Komornikova_Mesiar, Beatson_Ziegler}. In this section, we develop coordinate-wise monotonic ZFIFs without using differentiability of the multivariate ZFIFs on rectangular grids.
\begin{theorem}\label{2th7}
Let $f\in C(\mathcal{I})$ be non-zero and  increasing with respect to the variable $x_l$. Let 
\begin{gather*}
g_{\j}^{\epsilon}(X) :=f(u_{\j}^{\epsilon}(X)), \quad \gamma_{\j}^{\epsilon}:=\min_{X\in \mathcal{I}}\frac{\partial g_{\j}^{\epsilon}}{\partial x_l}(X),\\
\Gamma_{\j}^{\epsilon}:=\max_{X\in \mathcal{I}}\frac{\partial g_{\j}^{\epsilon}}{\partial x_l}(X), \quad \Gamma_{\n}:=\max_{X\in \mathcal{I}}\frac{\partial B_{\n}f}{\partial x_l}(X). 
\end{gather*}
Then, $f^{\alpha,\epsilon}_{\Delta;\n}(X)$ is increasing with respect to the variable $x_l$ if the partial derivative $f_{x_l}$ exists and the scaling functions $\alpha_{\j}$  given in \eqref{2eq015} satisfy the following conditions:\begin{equation}\label{2eq43}
\begin{split}
   & (i)\quad 0 \leq \alpha_{\j}(X)\leq \frac{\gamma_{\j}^{\epsilon}}{\Gamma_{\n}},\; \text{if  $\epsilon_{j_l}^l=0$ and $j_l$ odd, or, $\epsilon_{j_l}^l=1$ and $j_l$ even};\\ 
&(ii) \quad \frac{\Gamma_{\j}^{\epsilon}}{\Gamma_{\n}}\leq \alpha_{\j}(X)\leq 0,\; \text{if $\epsilon_{j_l}^l=0$ and $j_l$ even, or, $\epsilon_{j_l}^l=1$ and $j_l$  odd}.
\end{split}
\end{equation}
for $X \in \mathcal{I}$, $\j\in \prod\limits_{k=1}^{m}N_k$:

\end{theorem}
%
%
\begin{proof}
Let $X' :=(x_1,\dots,x_{l}',\dots,x_m), X'':=(x_1,\dots,x_{l},\dots,x_m)\in \prod\limits_{k=1}^{m}I_k$ where  $x_{l}'< x_{l}''$ and $f(X'')\geq f(X')$.
Then.
\begin{align*}
f^{\alpha,\epsilon}_{\Delta;\n}(u_{\j}^{\epsilon}(X'')) - f^{\alpha,\epsilon}_{\Delta;\n}(u_{\j}^{\epsilon}(X')) &= f(u_{\j}^{\epsilon}(X''))-f(u_{\j}^{\epsilon}(X'))+\alpha _{\j}(X)((f^{\alpha,\epsilon}_{\Delta;\n}(X'')\\
& \quad - (f^{\alpha,\epsilon}_{\Delta;n}(X'))(B_{\n}f(X''))-B_{\n}f(X'))).
 \end{align*}
 As $B_{\n}f$ is increasing with respect to the variable $x_l$ \cite{Foupouagnigni_Wouodjie}, $\Gamma_{\n}$ $>0$. Now there are two cases:
 \vskip 4pt\noindent
 \textbf{Case (i):} $\epsilon_{j_l}^l=0$ and $j_l$ odd, or, $\epsilon_{j_l}^l=1$ and $ j_l$ even, i.e.,  $u_l^{\epsilon^l}$ is increasing.
 
  In this case,    $f(u_{\j}^{\epsilon}(X))$ is increasing with respect to the variable $x_l$, which implies that $\gamma_{\textbf{j}}^{\epsilon}$ is non-negative.   Using the  mean value theorem for several variables applied to  $f(u_{j_1\cdots j_l\cdots j_m}^{\epsilon}(X''))-f(u_{j_1\cdots j_l\cdots j_m}^{\epsilon}(X'))$ and $(B_{\n}f(X'')-B_{\n}f(X'))$, yields
 \begin{align*}
& f^{\alpha,\epsilon}_{\Delta;\n}(u_{j_1\cdots j_l\cdots j_m}^{\epsilon}(X'')) - f^{\alpha,\epsilon}_{\Delta;\n}(u_{j_1,\cdots j_l\cdots ,j_m}^{\epsilon}(X'))\\ 
& \quad \geq \gamma_{j_1\cdots j_l\cdots j_m}^{\epsilon}(x_{l}''-x_{l}') - \alpha _{j_1\cdots j_l\cdots j_m}(X)\Gamma_{n}(x_{l}''-x_{l}')+\alpha _{j_1\cdots j_l\cdots j_m}(X)\\
&\qquad \cdot(f^{\alpha,\epsilon}_{\Delta;\n}(X'')-f^{\alpha,\epsilon}_{\Delta;\n}(X'))\\  &\quad = (\gamma_{j_1\cdots j_l\cdots j_m}^{\epsilon}-\alpha _{j_1\cdots j_l\cdots j_m}(X) \Gamma_{\n})(x_{l}'' - x_{l}')\\
& \qquad + \alpha_{j_1\cdots j_m}(X)(f^{\alpha,\epsilon}_{\Delta;\n}(X'')-f^{\alpha,\epsilon}_{\Delta;\n}(X')).
\end{align*}
If $\alpha_{j_1\cdots j_m}(X)\geq0$, then we need $\gamma_{j_1\cdots j_l\cdots j_m}^{\epsilon}-\alpha _{j_1\cdots j_l\cdots j_m}(X)\Gamma_{\n})(x_{l}''-x_{l}') \geq 0$ which yields the first condition in \eqref{2eq43}.
 \vskip 4pt\noindent
\textbf{Case (ii):} $\epsilon_{j_l}^l=0$ and $j_l$ even, or, $\epsilon_{j_l}^l=1$ and $ j_l$ odd, i.e., $u_{l,j_l}^{\epsilon^l}$ is decreasing.

In this case,  $f(u_{j_1\cdots j_l\cdots j_m}^{\epsilon}(X))$ is decreasing with respect to the variable $x_l$, which ensures that $\gamma_{j_1\cdots j_l\cdots j_m}^{\epsilon}$ is non-positive.   

If $\alpha_{j_1\cdots  j_m}(X)\leq0$, then an application of the mean value theorem for several variables applied to $f(u_{j_1\cdots j_l\cdots j_m}^{\epsilon}(X''))-f(u_{j_1\cdots j_l\cdots j_m}^{\epsilon}(X'))$ and $(B_{\n}f(X'')-B_{\n}f(X'))$, yields
 \begin{align*}
  f^{\alpha,\epsilon}_{\Delta;\n}(u_{j_1,\cdots ,j_l,\cdots ,j_m}^{\epsilon}(X''))&- f^{\alpha,\epsilon}_{\Delta;\n}(u_{j_1,\cdots, j_l,\cdots, j_m}^{\epsilon}(X'))\\ 
  &\leq \Gamma_{j_1,\cdots ,j_l,\cdots ,j_m}^{\epsilon}(x_{l}'' - x_{l}') - 
  \alpha _{j_1,\cdots ,j_l,\cdots , j_m}(X)\Gamma_{\n}(x_{l}''-x_{l}')\\
  & \quad + \alpha _{j_1,\cdots ,j_,\cdots , j_m}(X)(f^{\alpha,\epsilon}_{\Delta;\n}(X'')-f^{\alpha,\epsilon}_{\Delta;n}(X'))\\  
  & = (\Gamma_{j_1;\cdots ,j_l;\cdots , j_m}^{\epsilon}-\alpha _{j_1,\cdots ,j_l,\cdots ,j_m}(X)\Gamma_{\n})(x_{l}''-x_{l}')\\ 
  & \quad + \alpha _{j_1,\cdots ,j_m}(X)(f^{\alpha,\epsilon}_{\Delta;\n}(X'')-f^{\alpha,\epsilon}_{\Delta;\n}(X')).
\end{align*}
Thus, $\Gamma_{j_1,\cdots ,j_l,\cdots ,j_m}^{\epsilon}-\alpha _{j_1,\cdots ,j_l,\cdots ,j_m}(X)  \Gamma_{\n})(x_{l}''-x_{l}')\leq0$ if the second inequality in \eqref{2eq43} is true.
\end{proof}
Since fractal interpolation is an iterative process, it ensures that $f^{\alpha,\epsilon}_{\Delta;\n}$ is increasing  with respect to the variable $x_l$.

\begin{remark}
 Using similar arguments, we can construct coordinate-wise monotonically decreasing multivariate Bernstein zipper  $\alpha$-fractal functions $f^{\alpha,\epsilon}_{\Delta;\n}(X)$ for coordinate-wise monotonically decreasing functions $f\in C(\mathcal{I})$.
\end{remark} 

\section{Box Dimension of Multivariate ZFIF}
In this section, we derive bounds for the box dimension of the graph of a multivariate ZFIF. Furthermore, we show that a multivariate Bernstein polynomial $B_{\n}f$ is H\"olderian with exponent $\beta$ provided that $f$ is  H\"olderian with exponent $\beta$. This will be used to obtain estimates for the box dimension of graphs of multivariate zipper Bernstein fractal functions.

\begin{definition}
Let  $A\in \mathbb{R}_0^+$ and $0<\beta \leq 1$. Then,  $\Lip_A \beta$ is defined as the set of all functions $f: \mathcal{K}\subset \mathbb{R}^m\to \mathbb{R} $ satisfying 
\[
|f(X_2)-f(X_1)|\leq A\|X_2-X_1\|^{\beta}, \quad \forall\, X_1,X_2 \in \mathcal{K}.
\]
Such functions are also called uniformly H\"olderian with exponent $\beta$.
\end{definition}
In the next theorem, we provide  estimates for the fractal dimension of the graph of  a multizipper FIF. For this purpose, we use uniform partitions of $I_k=[0,1]$, $k\in \mathbb{N}_m$.
Based on the structure of the IFSs \eqref{2eq000} and  \eqref{2eq23},  we choose $u_{k,j_k}^{\epsilon^k}:I_k \to I_{k,j_k}$ as
\begin{align}
   u_{k,j_k}^{\epsilon^k}(x_k) := 
   \begin{cases}
    \frac{1-2\epsilon^k_{j_k}}{N_k}x_k+\frac{j_k-1+\epsilon^k_{j_k}}{N_k}, & \text{if}\; j_k \;\text{is odd};\\
    \frac{-1+2\epsilon^k_{j_k}}{N_k}x_k+\frac{j_k-\epsilon^k_{j_k}}{N_k}, & \text{if}\; j_k \;\text{is even}.
    \end{cases}
     \quad j_k \in \mathbb{N}_{N_k},\; k\in \mathbb{N}_m.
\end{align}

\begin{definition}{\cite{Falconer}}
Let $A $ be a non-empty bounded subset of $\mathbb{R}^n$. Suppose $\Lambda(\delta)$ denote the smallest number of $m$-dimensional cube of side $\delta$  that can cover $A$. The lower and upper box-counting dimensions of $A$ are defined as
\begin{align*}
    \begin{split}
     {\lowdim}(A)=\lowlim_{\delta \to 0}\frac{\log(N_{\delta}(A))}{-\log(\delta)}\\
     {\updim}(A)=\uplim_{\delta \to 0}\frac{\log(N_{\delta}(A))}{-\log(\delta)},
    \end{split}
\end{align*}
respectively. If these are equal and finite, we refer to the common value as the box-counting dimension
or box dimension of $A$:
\begin{align*}
    {\dim}_B(A)=\lim_{\delta \to 0}\frac{\log(N_{\delta}(A))}{-\log(\delta)}.
\end{align*}
\end{definition}
Suppose that the IFS  \eqref{2eq000} generates a multizipper FIF $f_{\n}^{(\alpha,\epsilon)}$. Then, we have the following result.
\begin{theorem}\label{Thm-BD}
Let $f,b\in C(\mathcal{I})$ with H\"older exponents $\xi_1, \xi_2 \in (0,1]$. Let $G$ be the graph of the fractal function $f_{\n}^{(\alpha,\epsilon)}$ associated with the IFS \eqref{2eq000}. Suppose that the interpolation points  are not contained in an $(m-1)$-dimensional hyperplane of $\R^m$. Let $\xi :=\min\{\xi_1,\xi_2\}$ and let 
\[
\gamma := \sum\limits_{j_1=1}^{N_1} \sum\limits_{j_1=2}^{N_2} \cdots \sum\limits_{j_k=1}^{N_k} \|\alpha_{\textbf{j}}\|_{\infty},
\]
Then, we have the following bounds for the box dimension of  $G$ based on the magnitude of $\gamma$:
\begin{itemize}
    \item[(i)]  If $\gamma \leq 1$, then  $m\leq dim_B(G) \leq m+1-\xi$;
    \item[(ii)] If  $\gamma > 1$ and  $(N_1N_2\cdots N_m)^{(\xi-m)} \gamma \leq 1$, then 
    \[
    m\leq dim_B(G) \leq m+1-\xi + \frac{\log(\gamma)}{\log(N_1N_2\cdots N_m)};
    \]
    \item[(iii)] If  $\gamma > 1$ and  $(N_1N_2\cdots N_m)^{(\xi-m)} \gamma > 1$, then 
    \[
    m\leq dim_B(G) \leq  1+ \frac{\log(\gamma)}{\log(N_1N_2\cdots N_m)}.
    \]
\end{itemize}
\end{theorem}
\begin{proof}
Our aim is to calculate the box dimension of the graph of the fractal function $f_{\n}^{(\alpha,\epsilon)}$. For this, we consider a cover $\Lambda(r)$ of $G$  whose elements are $m$-cubes  with  sides of length $\frac{1}{(N_1N_2\cdots N_m)^r}$ and of the form 
\begin{align}\label{eqdim1}
 \begin{split}
      \left[\frac{p_1-1}{(N_1N_2\cdots N_m)^r}, \frac{p_1}{(N_1N_2\cdots N_m)^r} \right] \times  \left[\frac{p_2-1}{(N_1N_2\cdots N_m)^r}, \frac{p_2}{(N_1N_2\cdots N_m)^r} \right]\times\\
      \cdots \times \left[\frac{k_m-1}{(N_1N_2\cdots N_m)^r}, \frac{p_m}{(N_1N_2\cdots N_m)^r} \right] \times \left[c, c+ \frac{1}{(N_1N_2\cdots N_m)^r} \right],
    \end{split}
\end{align}
{for} $p_i = 1, 2, \ldots, (N_1N_2\cdots N_m)^r$, $r\in\N_0$, $i\in \mathbb{N}_m$, and $c\in \mathbb{R}$.
 Suppose  $\mathcal{N}(r)$ is the number of such cubes necessary to cover the graph $G$.
 Let $\mathcal{N}_0(r)$ be the smallest number of arbitrary $(m+1)$-cubes of size
 \[
 \prod_{j=1}^{m+1} \frac{1}{(N_1N_2\cdots N_m)^r}.
 \]
 required to cover $G$. Hence, $\mathcal{N}_0(r)\leq \mathcal{N}(r)$. 
 
 Each arbitrary $(m+1)$-dimensional cube can be cover by at most $2^m$ $(m+1)$-cube of the form \eqref{eqdim1}. Thus, $\mathcal{N}(r)\leq 2^m\mathcal{N}_0(r)$ and, therefore, 
 \[
 \mathcal{N}_0(r)\leq \mathcal{N}(r)\leq 2^m\mathcal{N}_0(r).
 \]
Hence, we can use covers of the  form \eqref{eqdim1} to compute the box dimension of the graph $G$ of $f_{\n}^{(\alpha,\epsilon)}$.
 
Denote by $\Lambda(r,p_1,p_2,\cdots,p_m)$ the collection of $(m+1)$-cubes in 
 \begin{align}\label{eqdim2}
 \begin{split}
      \left[\frac{p_1-1}{(N_1N_2\cdots N_m)^r}, \frac{p_1}{(N_1N_2\cdots N_m)^r} \right] \times  \left[\frac{p_2-1}{(N_1N_2\cdots N_m)^r}, \frac{p_2}{(N_1N_2\cdots N_m)^r} \right] \times \cdots \\
       \cdots \times \left[\frac{p_m-1}{(N_1N_2\cdots N_m)^r}, \frac{p_m}{(N_1N_2\cdots N_m)^r} \right], \text{ for } p_i=1,2,\cdots,(N_1N_2\cdots N_m)^r
    \end{split}
\end{align}
of the form \eqref{eqdim1} consitsting of $\mathcal{N}(r,p_1,p_2,\cdots,p_m)$ $(m+1)$-dimensional cubes.
One observes that
$$\mathcal{N}(r)=\sum_{p_1=1}^{(N_1N_2\cdots N_m)^r}  \sum_{p_2=1}^{(N_1N_2\cdots N_m)^r}\cdots \sum_{p_m=1}^{(N_1N_2\cdots N_m)^r} \mathcal{N}(r,p_1,p_2,\cdots,p_m).$$
For  $\textbf{j}\in \prod\limits_{k=1}^m \mathbb{N}_k$, the image of $\Lambda(r,p_1,p_2,\cdots,p_m)$ under the map $v_{\textbf{j}}$ is contained in 
\begin{align}\label{eqdim3}
 \begin{split}
      & \left[\frac{N_2 N_3\cdots N_m (l_1(p_1,j_1)-1)}{(N_1N_2\cdots N_m)^{r+1}}, \frac{N_2N_3\cdots N_m l_1(p_1,j_1)}{(N_1N_2\cdots N_m)^{r+1}} \right] \\  & \quad \times\left[\frac{N_1N_3\cdots N_m (l_2(p_2,j_2)-1)}{(N_1N_2\cdots N_m)^{r+1}}, \frac{N_1N_3\cdots N_m l_2(k_2,j_2)}{(N_1N_2\cdots N_m)^{r+1}} \right]\times\cdots \\
      & \quad\times \left[\frac{N_1N_3\cdots N_{m-1}(l_m(p_m,j_m)-1)}{(N_1N_2\cdots N_m)^{r+1}}, \frac{N_1N_3\cdots N_{m-1}l_m(p_m,j_m)}{(N_1N_2\cdots N_m)^{r+1}} \right] \times \mathbb{R},
    \end{split}
\end{align}
where $l_k(p_k,j_k)=p_k+(j_k-1)(N_1N_2\cdots N_m)^r$. Therefore, we obtain 
\begin{align}\label{eqdim4}
    \begin{split}
      \mathcal{N}(r+1,p_1,p_2,\cdots,p_m)&=\sum_{j_1=1}^{N_1}  \sum_{j_2=1}^{N_2}\cdots \sum_{j_m=1}^{N_m} \sum_{p_1,p_2,\cdots,p_m=1}^{(N_1N_2\cdots N_m)^r}\\
      &~~~~~~~~~~~~~\cdot\mathcal{N}(r,l_1(p_1,j_1),l_2(p_2,j_2),\cdots,l_k(p_k,j_k)).  
    \end{split}
\end{align}
As $f$ and $b$ are uniform H\"olderian on $I$ with exponents $\xi_1,\xi_2 \in (0,1]$, we obtain the following estimates for $X=(x_1,x_2,\cdots,x_m), X'=(x_1',x_2',\cdots,x_m') \in \prod\limits_{k=1}^{m}\left[ \frac{p_k-1}{(N_1N_2\cdots N_m)^{r}},\frac{p_k}{(N_1N_2\cdots N_m)^{r}} \right]$:
\begin{align}
    \begin{split}
       |f(u_{\textbf{j}}(X))-f(u_{\textbf{j}}(X'))| &\leq \frac{A_1}{(N_1N_2\cdots N_m)^{\xi_1(r+1)}},\\
      |b(X)-b(X')| &\leq \frac{A_2}{(N_1N_2\cdots N_m)^{\xi_2 r}}.
    \end{split}
\end{align}
Thus, the maximum height of $v_{\textbf{j}}(\Lambda(r,p_1,p_2,\cdots,p_m))$ is bounded above by
\[
\frac{\|\alpha_{\textbf{j}}\|_{\infty}\mathcal{N}(r,p_1,p_2,\cdots,p_m)}{(N_1N_2\cdots N_m)^{ r}}+\frac{A_1}{(N_1N_2\cdots N_m)^{\xi_1(r+1)}}+\frac{A_2\|\alpha_{\textbf{j}}\|_{\infty}}{(N_1N_2\cdots N_m)^{\xi_2 r}}.
\]
Now,
\begin{align*}
    \begin{split}
     \mathcal{N}(r,l_1(p_1,j_1),l_2(p_2,j_2),\cdots,l_k(p_k,j_k))= \left(\frac{\|\alpha_{\textbf{j}}\|_{\infty}\mathcal{N}(r,p_1,p_2,\cdots,p_m)}{(N_1N_2\cdots N_m)^{ r}}\right.\\ \left.+ \frac{A_1}{(N_1N_2\cdots N_m)^{\xi_1(r+1)}} + \frac{A_2\|\alpha_{\textbf{j}}\|_{\infty}}{(N_1N_2\cdots N_m)^{\xi_2 r}}\right) (N_1N_2\cdots N_m)^{r+1}+2\\
     = \|\alpha_{\textbf{j}}\|_{\infty} \mathcal{N}(r,p_1,p_2,\cdots,p_m) (N_1N_2\cdots N_m)+ A_1 (N_1N_2\cdots N_m)^{(1-\xi_1)(r+1)}\\
     + A_2\|\alpha_{\textbf{j}}\|_{\infty} (N_1N_2\cdots N_m)^{(1-\xi_2)r+1}+ 2.
    \end{split}
\end{align*}
This produces an estimate of the form
\begin{align}
    \begin{split}
       \sum_{j_1=1}^{N_1}  \sum_{j_2=1}^{N_2}&\cdots \sum_{j_m=1}^{N_m}\mathcal{N}(r,l_1(p_1,j_1),l_2(p_2,j_2),\cdots,l_k(p_k,j_k)) \leq\\
       &\mathcal{N}(r,p_1,p_2,\cdots,p_m) (N_1N_2\cdots N_m) \gamma + A_1 (N_1N_2\cdots N_m)^{(1-\xi_1)(r+1)+1}\\
       & \quad + A_2 (N_1N_2\cdots N_m)^{(1-\xi_2)r+1} \gamma + 2(N_1N_2\cdots N_m). 
    \end{split}
\end{align}
Substituting the above estimate into \eqref{eqdim4}, we obtain
\begin{align}
    \begin{split}
     \mathcal{N}(r+1)&=  \sum_{p_1,p_2,\cdots,p_m=1}^{(N_1N_2\cdots N_m)^r}  \sum_{j_1=1}^{N_1}  \sum_{j_2=1}^{N_2}\cdots \sum_{j_m=1}^{N_m}\mathcal{N}(r,l_1(p_1,j_1),l_2(p_2,j_2),\cdots,l_m(p_m,j_k))\\
     &\leq \sum_{p_1,p_2,\cdots,p_m=1}^{(N_1N_2\cdots N_m)^r} \mathcal{N}(r,p_1,p_2,\cdots,p_m) (N_1N_2\cdots N_m) \gamma \\
     & \quad + A_1 (N_1N_2\cdots N_m)^{(1-\xi_1)(r+1)+1}+A_2 (N_1N_2\cdots N_m)^{(1-\xi_2)r+1} \gamma \\
     & \quad + 2(N_1N_2\cdots N_m))\\
     &= \mathcal{N}(r) (N_1N_2\cdots N_m) \gamma + A_1 (N_1N_2\cdots N_m)^{(1-\xi_1)(r+1)+mr+1} \\
     & \quad + A_2 (N_1N_2\cdots N_m)^{(1-\xi_2)r+1+mr} \gamma + 2(N_1N_2\cdots N_m)^{m r+1})\\
     &\leq \mathcal{N}(r) (N_1N_2\cdots N_m) \gamma + A_1 (N_1N_2\cdots N_m)^{(r+1)(m+1-\xi)}\\
     & \quad + A_2 (N_1N_2\cdots N_m)^{(r+1)(m+1-\xi)} \gamma + 2(N_1N_2\cdots N_m)^{(r+1)(m+1-\xi)}\\
     &=\mathcal{N}(r) (N_1N_2\cdots N_m) \gamma + (N_1N_2\cdots N_m)^{(r+1)(m+1-\xi)} \mathcal{C}, 
    \end{split}
\end{align}
where $\mathcal{C}:=A_1+A_2 \gamma +2$.
Using the above inequality repeatedly with respect to $r$, we obtain the following geometric series type of expression 
\begin{align*}
    \begin{split}
       \mathcal{N}(r)&\leq \mathcal{N}(r-1) (N_1N_2\cdots N_m) \gamma + (N_1N_2\cdots N_m)^{r(m+1-\xi)} \mathcal{C}\\
   &\leq [\mathcal{N}(r-2) (N_1N_2\cdots N_m) \gamma + (N_1N_2\cdots N_m)^{(r-1)(m+1-\xi)} \mathcal{C}](N_1N_2\cdots N_m) \gamma \\
& \quad + (N_1N_2\cdots N_m)^{r(m+1-\xi)} \mathcal{C}\\
   &\leq \mathcal{N}(r-2) (N_1N_2\cdots N_m)^2 \gamma^2 + (1+(N_1N_2\cdots N_m)^{(\xi-m-1)} \gamma)\\
   & \quad\cdot(N_1N_2\cdots N_m)^{r(m-\xi)} \mathcal{C}.
    \end{split}
\end{align*}
Continuing this process yields
\begin{align}\label{eqdim5}
    \begin{split}
     \mathcal{N}(r)\leq   \mathcal{N}(0) (N_1N_2\cdots N_m)^{r} \gamma^{r} + \{1+(N_1N_2\cdots N_m)^{(\xi-m)} \gamma\\
     + (N_1N_2\cdots N_m)^{2(\xi-m)} \gamma^2+ \cdots +  (N_1N_2\cdots N_m)^{(r-1)(\xi-m)} \gamma^{(r-1)} \}\\
      \cdot (N_1N_2\cdots N_m)^{r(m+1-\xi)} \mathcal{C}
    \end{split}
\end{align}
Thus, we have the  following three cases:
\vskip 4pt\noindent
\textbf{Case (i):} $\gamma \leq 1$. 
\vskip 4pt\noindent
As $N_k \geq 2$ and $\xi \in (0,1]$, $(N_1N_2\cdots N_m)^{k(\xi-m)} \leq 1$, for $k\geq 1$. By \eqref{eqdim5}, we have
\begin{align}
    \begin{split}
     \mathcal{N}(r)&\leq   \mathcal{N}(0) (N_1N_2\cdots N_m)^{r} \gamma^{r} +  r(N_1N_2\cdots N_m)^{r(m+1-\xi)} \mathcal{C}\\
     &\leq \mathcal{N}(0) (N_1N_2\cdots N_m)^{r(m+1-\xi)} r  +  r(N_1N_2\cdots N_m)^{r(m+1-\xi)} \mathcal{C}\\
     &\leq \mathcal{C}_1 r (N_1N_2\cdots N_m)^{r(m+1-\xi)}, \quad \text{where}\quad \mathcal{C}_1=\mathcal{N}(0)+\mathcal{C}.
    \end{split}
\end{align}
Hence, 
\begin{align}\label{eqdim6}
    \begin{split}
     \dim_B(G)&=   \lim_{r \to \infty} \frac{\log(\mathcal{N}(r))}{\log((N_1N_2\cdots N_m)^r)}\\
     &\leq\lim_{r \to \infty} \frac{\log(\mathcal{C}_1 r (N_1N_2\cdots N_m)^{r(m+1-\xi)})}{\log((N_1N_2\cdots N_m)^r)}
     =m+1-\xi.
    \end{split}
\end{align}
By the continuity of the fractal function, we have that $\dim_B(G) \geq m$, and  using \eqref{eqdim6}, we obtain
\[
m\leq \dim_B(G) \leq m+1-\xi.
\]
\vskip 4pt\noindent
 \textbf{Case (ii):} $\gamma >1$ and $(N_1N_2\cdots N_m)^{(\xi-m)} \gamma \leq 1$.
 \vskip 4pt\noindent
 By \eqref{eqdim5}, we have
 \begin{align}
 \begin{split}
   \mathcal{N}(r)&\leq \mathcal{N}(0) (N_1N_2\cdots N_m)^{r} \gamma^{r} + r(N_1N_2\cdots N_m)^{r(m+1-\xi)} \mathcal{C}\\ 
   &\leq \mathcal{C}_2  \gamma^{r} r (N_1N_2\cdots N_m)^{r(m+1-\xi)}, \quad \text{where} \quad \mathcal{C}_2 = \mathcal{C} + \mathcal{N}(0).
 \end{split}  
 \end{align}
  Hence 
  \begin{align}
    \begin{split}
     \dim_B(G)& \leq   \lim_{r \to \infty} \frac{\log(\mathcal{C}_2 \gamma^{r} r (N_1N_2\cdots N_m)^{r(m+1-\xi)})}{\log(N_1N_2\cdots N_m)^r}\\
     &=m+1-\xi+ \frac{\log(\gamma)}{\log(N_1N_2\cdots N_m)}.
    \end{split}
\end{align}
Hence, using the above inequality, we obtain 
\[
m\leq \dim_B(G) \leq m+1-\xi + \frac{\log(\gamma)}{\log(N_1N_2\cdots N_m)}.
\]
\vskip 4pt\noindent
\textbf{Case (iii)}: $\gamma >1$ and $(N_1N_2\cdots N_m)^{(\xi-m)} \gamma > 1$. 
\vskip 4pt\noindent
Again by \eqref{eqdim5}
\begin{align}\label{eqbox10}
 \begin{split}
   \mathcal{N}(r)&\leq \mathcal{N}(0) (N_1N_2\cdots N_m)^{r} \gamma^{r} + \left[ \frac{(N_1N_2\cdots N_m)^{r(\xi-m)} \gamma^r-1}{(N_1N_2\cdots N_m)^{(\xi-m)} \gamma-1} \right]\\
   & \quad\cdot(N_1N_2\cdots N_m)^{r(m+1-\xi)} \mathcal{C}\\ 
  &\leq \mathcal{N}(0) (N_1N_2\cdots N_m)^{r} \gamma^{r} + \left[ \frac{(N_1N_2\cdots N_m)^{r(\xi-m)} \gamma^r}{(N_1N_2\cdots N_m)^{(\xi-m)} \gamma-1} \right]\\
& \quad \cdot(N_1N_2\cdots N_m)^{r(m+1-\xi)} \mathcal{C}\\
  &\leq \mathcal{N}(0) (N_1N_2\cdots N_m)^{r} \gamma^{r} +\frac{(N_1N_2\cdots N_m)^r \gamma^r}{(N_1N_2\cdots N_m)^{(\xi-m)} \gamma-1}\,\mathcal{C}.
 \end{split}  
 \end{align}
Using \eqref{eqbox10} and arguments similar to those above, we obtain 
\begin{align*}
    \begin{split}
     \dim_B(G) \leq  1+ \frac{\log(\gamma)}{\log(N_1N_2\cdots N_m)}.
    \end{split}
\end{align*}
\end{proof}
 \begin{remark}
The cases (i)-(iii) in Theorem \ref{Thm-BD} imply that the box dimension of ZFIFs is independent of the signature matrix $\epsilon$.
\end{remark} 

It is known that for a univariate function $f\in \Lip_A \beta$, the corresponding univariate Bernstein function $B_nf \in \Lip_A \beta$\cite{Brown_Elliott_Paget}. We need a similar result for the computation of the box dimension of the graph of a multivariate Bernstein zipper fractal function.
\begin{proposition}\label{lemmadim}
If a multivariate function $f$ defined on $[0,1]^m$ is in $\Lip_A \beta$ then the corresponding multivariate Bernstein polynomial $B_{\n}f$ is also an element of $\Lip_A \beta$.
\end{proposition}
\begin{proof}
We know that 
\begin{align}
    \begin{split}
    B_{\n}f(X) = \sum_{j_1=0}^{n_1} \sum_{j_2=0}^{n_2} \cdots \sum_{j_m=0}^{n_m}  \prod\limits_{p=1}^{m} b_{j_p,n_p}(x_p) f\left(\frac{k_1}{n_1},\frac{k_2}{n_2},\cdots,\frac{k_m}{n_m}\right), 
    \end{split}
\end{align}
where $b_{j_p,n_p}(x_p):=\binom{n_p}{k_p}x^{k_p}(1-x_p)^{n_p-k_p}$, $p\in \mathbb{N}_m$, and $X:=(x_1,x_2,\cdots,x_m)$. Let 
\begin{align}
      A_{k_p,l_p}^{n_p}(x_p,y_p):=\frac{n_p!}{k_p! l_p!(n_p-k_p- l_p)!}\, x_p^{k_p}(y_p-x_p)^{l_p}(1-y_p)^{n_p-k_p-l_p}
\end{align}
for $p\in \mathbb{N}_m$ and $x_p,y_p \in I$. 

The following results is valid for $n\in \mathbb{N}$:
\begin{align}\label{eqlip03}
       \sum_{k=0}^{n} \binom{n}{k} & y^{k}(1-y)^{n-k} f\left(\frac{k}{n}\right)\nonumber\\
       & = \sum_{k=0}^{n} \sum_{l=0}^{n-k} \frac{n!}{k! l!(n-k- l)!}\; x^{k}(y-x)^{l}(1-y)^{n-k-l}f\left(\frac{k}{n}\right).
 \end{align}
Let $X:=(x_1,x_2,\cdots,x_m)$ and $Y:=(y_1,y_2,\cdots,y_m)$. Then, there are $2^m$ possible arrangements in the corresponding arguments of $X$ and  $Y$.  Take as one possible case
$x_1\leq y_1$, $x_2\geq y_2$, $x_3\geq y_3$, \ldots, $x_m\geq y_m$. Eqn.  \eqref{eqlip03} implies
\begin{align*}
    \begin{split}
     B_{\n}f(X) &=   \sum_{j_2=0}^{n_2} \cdots \sum_{j_m=0}^{n_m}  \prod\limits_{p=1,p \ne 1}^{m} b_{j_p,n_p}(x_p) \sum_{j_1=0}^{n_1} b_{j_1,n_1}(x_1)f\left(\frac{j_1}{n_1},\frac{j_2}{n_2},\cdots,\frac{j_m}{n_m}\right)\\
    &=\sum_{j_2=0}^{n_2} \cdots \sum_{j_m=0}^{n_m}  \prod\limits_{p=1,p \ne 1}^{m} b_{j_p,n_p}(x_p) \sum_{k_1=0}^{n_1} \sum_{l_1=0}^{n_1-k_1} A_{j_1,l_1}^{n_1}(y_1,x_1)\\ &~~~~~~~~~~~~~~~~~~~~~~~~~~~~~~~~~~~~~~~~~~~~~~~~~~~~~\cdot f\left(\frac{k_1}{n_1},\frac{j_2}{n_2},\cdots,\frac{j_m}{n_m}\right)\\
    &=\sum_{k_1=0}^{n_1} \sum_{l_1=0}^{n_1-k_1} \sum_{j_3=0}^{n_3} \cdots \sum_{j_m=0}^{n_m}  \prod\limits_{p=1,p \ne 1}^{m} b_{j_p,n_p}(x_p)  A_{j_1,l_1}^{n_1}(y_1,x_1) \\ 
    &~~~~~~~~~~~~~~~~~~~~~~~~~~~~~~~~~~~~\cdot \sum_{j_2=0}^{n_2} b_{j_2,n_2}(x_2)f\left(\frac{k_1}{n_1},\frac{j_2}{n_2},\cdots,\frac{j_m}{n_m}\right)\\
   & =\sum_{k_1=0}^{n_1} \sum_{l_1=0}^{n_1-k_1} \sum_{j_3=0}^{n_3} \cdots \sum_{j_m=0}^{n_m}  \prod\limits_{p=1,p \ne 1, 2}^{m} b_{j_p,n_p}(x_p)  A_{j_1,l_1}^{n_1}(y_1,x_1)\\
   &~~~~~~~~~~~~~~~~~~~~~~\cdot \sum_{k_2=0}^{n_2} \sum_{l_2=0}^{n_2-k_2} A_{j_2,l_2}^{n_2}(x_2,y_2) f\left(\frac{k_1}{n_1},\frac{k_2+l_2}{n_2},\cdots,\frac{j_m}{n_m}\right),
      \end{split} 
\end{align*}
    \begin{align*}
    \begin{split}
     B_{\n}f(X) &=\sum_{k_1=0}^{n_1} \sum_{l_1=0}^{n_1-k_1}  \sum_{k_2=0}^{n_2} \sum_{l_2=0}^{n_2-k_2}  \sum_{j_3=0}^{n_3} \cdots \sum_{j_m=0}^{n_m}  \prod\limits_{p=1,p \ne 1, 2}^{m} b_{j_p,n_p}(x_p)\\  
    &~~~~~~~~~~~~~~~~~~\cdot A_{j_1,l_1}^{n_1}(y_1,x_1) A_{j_2,l_2}^{n_2}(x_2,y_2)f\left(\frac{k_1}{n_1},\frac{k_2+l_2}{n_2},\cdots,\frac{j_m}{n_m}\right).
    \end{split} 
\end{align*}
Continuing this process, we obtain
\begin{align}\label{eqlip3}
    \begin{split}
     B_{\n}f(X)=\sum_{k_1=0}^{n_1} \sum_{l_1=0}^{n_1-k_1}  \sum_{k_2=0}^{n_2} \sum_{l_2=0}^{n_2-k_2}  \sum_{j_3=0}^{n_3} \cdots \sum_{k_m=0}^{n_m} \sum_{l_m=0}^{n_m-k_m}   A_{j_1,l_1}^{n_1}(y_1,x_1)\\
     \prod\limits_{p=1,p \ne 1}^{m}A_{j_p,l_p}^{n_p}(x_p,y_p)f\left(\frac{k_1}{n_1},\frac{k_2+l_2}{n_2},\cdots,\frac{k_m+l_m}{n_m}\right).
    \end{split} 
\end{align}
Similarly, we get 
\begin{align}\label{eqlip4}
    \begin{split}
     B_{\n}f(Y)=\sum_{k_1=0}^{n_1} \sum_{l_1=0}^{n_1-k_1}  \sum_{k_2=0}^{n_2} \sum_{l_2=0}^{n_2-k_2}  \sum_{j_3=0}^{n_3} \cdots \sum_{k_m=0}^{n_m} \sum_{l_m=0}^{n_m-k_m}   A_{j_1,l_1}^{n_1}(y_1,x_1)\\
     \prod\limits_{p=1,p \ne 1}^{m}A_{j_p,l_p}^{n_p}(x_p,y_p)f\left(\frac{k_1+l_1}{n_1},\frac{k_2}{n_2},\cdots,\frac{k_m}{n_m}\right).
    \end{split} 
\end{align}
By \eqref{eqlip3} and \eqref{eqlip4}, we can write
\begin{align}\label{eqlip5}
\begin{split}
     |B_{\n}f(X)&- B_{\n}f(Y)|\\
     &\leq\sum_{k_1=0}^{n_1} \sum_{l_1=0}^{n_1-k_1}  \sum_{k_2=0}^{n_2} \sum_{l_2=0}^{n_2-k_2}  \sum_{j_3=0}^{n_3}\cdots
     \sum_{k_m=0}^{n_m} \sum_{l_m=0}^{n_m-k_m} A_{j_1,l_1}^{n_1}(y_1,x_1)\\
     &~~~~~~~~~~~~~~~\cdot\prod\limits_{p=1,p \ne 1}^{m}A_{j_p,l_p}^{n_p}(x_p,y_p) | f\left(\frac{k_1}{n_1},\frac{k_2+l_2}{n_2},\cdots,\frac{k_m+l_m}{n_m}\right)\\
     &~~~~~~~~~~~~~~~~~~~~~~~~~~~~~~~~~~~~~~~~~-f\left(\frac{k_1+l_1}{n_1},\frac{k_2}{n_2},\cdots,\frac{k_m}{n_m}\right)|\\
     &\leq \sum_{k_1=0}^{n_1} \sum_{l_1=0}^{n_1-k_1}  \sum_{k_2=0}^{n_2} \sum_{l_2=0}^{n_2-k_2}  \sum_{j_3=0}^{n_3}\cdots \sum_{k_m=0}^{n_m} \sum_{l_m=0}^{n_m-k_m} A_{j_1,l_1}^{n_1}(y_1,x_1)\\
     &~~~~~~~~~~~~~~~\cdot \prod\limits_{p=1,p \ne 1}^{m}A_{j_p,l_p}^{n_p}(x_p,y_p) A \left(\max \left\{\frac{k_p}{n_p}:p\in \mathbb{N}_m\right\}\right)^{\beta}.
    \end{split}
    \end{align}
    Suppose $\max \left\{\frac{k_p}{n_p}:p\in \mathbb{N}_m\right\}=\frac{k_a}{n_a}$, for some $a\in \mathbb{N}_m$.
    Then, \eqref{eqlip5} becomes
    \begin{align}\label{eqlip6}
\begin{split}
     |B_{\n}f(X)&- B_{\n}f(Y)|
     \leq \sum_{k_1=0}^{n_1} \sum_{l_1=0}^{n_1-k_1}  \sum_{k_2=0}^{n_2} \sum_{l_2=0}^{n_2-k_2}  \\
     &~~~~~~~~~~\cdots \sum_{k_m=0}^{n_m} \sum_{l_m=0}^{n_m-k_m} A_{j_1,l_1}^{n_1}(y_1,x_1)\prod\limits_{p=1,p \ne 1}^{m}A_{j_p,l_p}^{n_p}(x_p,y_p)  \left(\frac{k_a}{n_a}\right)^{\beta}\\
     &=A \sum_{k_1=0}^{n_1} \sum_{l_1=0}^{n_1-k_1} A_{j_1,l_1}^{n_1}(y_1,x_1) \sum_{k_2=0}^{n_2} \sum_{l_2=0}^{n_2-k_2} A_{j_2,l_2}^{n_2}(x_2,y_2)\\
     &\cdots \sum_{k_a=0}^{n_a} \sum_{l_a=0}^{n_a-k_a} A_{j_a,l_a}^{n_a}(x_a,y_a)\left(\frac{k_a}{n_a}\right)^{\beta} \cdots  \sum_{k_m=0}^{n_m} \sum_{l_m=0}^{n_m-k_m} A_{j_m,l_m}^{n_m}(x_m,y_m)\\
     &=A \sum_{l_1=0}^{n_1} b_{l_1,n_1}(x_1) \cdots \sum_{l_a=0}^{n_a} b_{l_a,n_a}(x_a) \left(\frac{k_a}{n_a}\right)^{\beta} \cdots \sum_{l_m=0}^{n_m} b_{l_m,n_m}(x_m)\\
     &=A \sum_{l_a=0}^{n_a} b_{l_a,n_a}(x_a) \left(\frac{l_a}{n_a}\right)^{\beta}= A B_{n_a}(x^{\beta},(y_a-x_a))\leq A (y_a-x_a)^{\beta}\\
     &= A \|X-Y\|^{\beta}.
    \end{split}
    \end{align}
Similarly, in all other cases, we get the same inequality. Therefore, 
    \[B_{\n}f \in \Lip_A^{\beta}.\]
\end{proof}

\begin{corollary}
Let $f\in C(\mathcal{I})$ with H\"older exponent $\xi \in (0,1]$. Let $G$ be the graph of the multivariate Bernstein fractal function $f_{\n}^{(\alpha,\epsilon)}$ associated with the IFS \eqref{2eq23}. Suppose that the interpolation points are not contained in an $(m-1)$ dimensional hyperplane. Let
\[
 \gamma := \sum\limits_{j_1=1}^{N_1} \sum\limits_{j_1=2}^{N_2} \cdots \sum\limits_{j_k=1}^{N_k} \|\alpha_{\textbf{j}}\|_{\infty}.
\]
Then,  the box dimension of  $G$ satisfies the estimates (i), (ii), (iii) of  Theorem \ref{Thm-BD}.
\end{corollary}
\begin{proof}
As $f$ is H\"olderian with exponent $\xi$, Lemma \ref{lemmadim} ensures that $B_{\textbf{n}}f$ is H\"olderian with the same exponent. Thus, the results in Theorem \ref{Thm-BD} are also valid for the box dimension of the graphs of multivariate Bernstein fractal functions.
\end{proof}
\section{Conclusions}  In this work, we have introduced multivariate zipper fractal interpolation prescribed on multivariate data given on a Cartesian grid by a binary signature matrix. Multivariate zipper $\alpha$-fractal functions are constructed and its approximation properties studied. Employing in the construction a multivariate Bernstein function as the base function, we have studied some shape preserving aspects of multivariate Bernstein zipper $\alpha$-fractal functions. Finally, we have derived bounds for box-dimension of the graph of a multivariate zipper  $\alpha$-fractal function based on the scaling factors and the H\"older exponents of a given function and base function. It was found that our methodology provides $2^m$-iso-dimensional multivariate fractal functions for the same scaling factor. 
\section*{Acknowledgement} The  second author would like to thank
IIT Madras for funding from the IoE project: SB20210848MAMHRD008558 through Ministry of Education, Govt. of India.

\end{document}